\documentclass[12pt]{article}

\usepackage{latexsym,amssymb,amsmath}
\usepackage[dvips]{graphicx}
\usepackage{bm}
\newcommand{\Z}{\mathbb Z}
\newcommand{\N}{\mathbb N}
\newcommand{\R}{\mathbb R}
\newcommand{\Q}{\mathbb Q}

\newcommand{\sma}{\left(\begin{array}}
\newcommand{\fma}{\end{array}\right)}

\newtheorem{lem}{Lemma}[section]
\newtheorem{defn}[lem]{Definition}
\newtheorem{ex}[lem]{Example}
\newtheorem{co}[lem]{Corollary}
\newtheorem{thm}[lem]{Theorem}
\newtheorem{prop}[lem]{Proposition}

\newenvironment{proof}{\textbf{Proof.}}{\newline\hspace*{\fill}{$\Box$}\\}

\begin{document}

\title{Quasi-actions whose quasi-orbits are quasi-isometric to trees}

\author{J.\,O.\,Button}
\date{}
\newcommand{\Address}{{% additional braces for segregating \footnotesize
  \bigskip
  \footnotesize

\textsc{Selwyn College, University of Cambridge,
Cambridge CB3 9DQ, UK}\par\nopagebreak
  \textit{E-mail address}: \texttt{j.o.button@dpmms.cam.ac.uk}
}}

\maketitle
\begin{abstract}
  We give necessary and sufficient conditions under which a quasi-action of any
  group on an arbitrary metric space can be reduced to a cobounded isometric
  action on some bounded valence tree, following a result of Mosher, Sageev
  and Whyte. Moreover if the quasi-action is metrically proper and
  quasi-orbits are quasi-isometric to trees then the group is virtually free.
\end{abstract}

\section{Introduction}
Let us say that a quasi-tree is any geodesic metric space which is
quasi-isometric to some simplicial tree. Suppose that a finitely
generated group $G$
acts on a quasi-tree then there is no reason to suppose that $G$
will act on a tree. Of course here we need to be more specific because
any group acts by isometries on any metric space via the trivial action.
Hence let us say for now that by ``acts'' in the above we mean acts by
isometries with unbounded orbits. Examples where $G$ acts on a
quasi-tree but not on a tree were first given in \cite{mnlms}. For more
recent examples, if we take $G$ to be
a non elementary hyperbolic group with property (T) then $G$ will have Serre's
property (FA), so that every action on a tree will have a global fixed
point. However by \cite{balas} we have that any non elementary hyperbolic
group, indeed any acylindrically hyperbolic group, has an unbounded
(acylindrical) action on a quasi-tree that is a graph $\Gamma$.
Indeed $\Gamma$ is a Cayley graph with respect to a generating set of $G$,
though this generating set will nearly always be infinite.

This suggests introducing some sort of finiteness condition. What if $G$
acts on a quasi-tree $X$ that is a graph and is of bounded valence? Will
$G$ then act on some simplicial tree $T$ and can we further take $T$ itself
to have bounded valence? A moment's consideration reveals no counterexamples,
even if we weaken the bounded valence condition on $X$ to local
finiteness or furthermore if we do not insist that $X$ is a graph but
just a quasi-tree which is a proper metric space. If however we require $X$
only to be a quasi-tree that is quasi-isometric to some proper metric space
then \cite{abos} Lemma 4.15 yields a counterexample: we take a non trivial
homogeneous quasi-morphism $h$ of a hyperbolic group $G$ having property
(T), whereupon we obtain from $h$ an infinite generating set $S$ such that
the Cayley graph of $G$ with respect to $S$ is quasi-isometric to $\R$.

A more useful version of this question is whether, given a group $G$ with
an action on a geodesic metric space $X$ with some finiteness condition
on $X$, this particular action can be changed into a similar action on
a (preferably bounded valence) tree $T$. This requires explaining what we
mean by ``changed into'' and ``similar'' but a standard notion in the
literature is that of a quasi-conjugacy: a quasi-isometry $F$ from
$X$ to $T$ (so that $X$ has to be a quasi-tree) which is coarse
$G$-equivariant, meaning that $F$ need not be exactly equivariant with
respect to the actions of $G$ on $X$ and on $T$, but which is equivariant
up to bounded error. Here we will say that these actions are equivalent.
Equivalence will preserve coarse properties of an action, for instance
being unbounded, being cobounded or being metrically proper.

Even more generally, we can consider a quasi-action of a group $G$ on a metric
space $X$. This can be thought of as an action up to bounded error and where
the self-maps $A_g:X\rightarrow X$ given by the elements $g\in G$ need not
be bijections but are quasi-isometries with uniform constants. We can then
generalise the definitions of these
coarse properties of actions as mentioned above to quasi-actions.
We can also do this for the concept of equivalence and these properties
will still be preserved under this more general notion. A disadvantage
is that quasi-actions will not behave as nicely as actions, for instance
we do not really have a good equivalent of the stabiliser of a point or
of a set. We can however generalise the notion of the orbit $Orb_G(x)$
of a point $x\in X$ to the case of a quasi-action: the definition
${\mathcal Q}_G(x)=\{A_g(x)\,|\,g\in G\}$ of the quasi-orbit
of $x$ is natural, and allows us to extend the notion of coboundedness
to quasi-actions: any quasi-orbit is coarse dense in $X$. 
However we will have to accept peculiarities such as $x$ need not
lie in its own quasi-orbit. The advantage of quasi-actions is that they
are very flexible: in particular if $G$ quasi-acts on $X$ and we have
a space $Y$ which is quasi-isometric to $X$ then there is an equivalent
quasi-action of $G$ on $Y$, which is certainly not true for actions.
Conversely if $G$ quasi-acts on any (geodesic) metric space $X$ then
it is true that there
is some (geodesic) metric space $Z$ and an equivalent isometric action  
of $G$ on $Z$. However, although $X$ and $Z$ will necessarily be
quasi-isometric, we cannot insist on other properties continuing to hold:
for instance if $X$ is proper and/or a simplicial tree then we cannot
assume that $Z$ will be either of these. Thus given the question:
``if $G$ acts on a quasi-tree then is there an equivalent action on a tree?''
the hypothesis is interchangeable with ``if $G$ quasi-acts on a tree''
(and indeed with ``if $G$ quasi-acts on a quasi-tree'') but on adding
any finiteness conditions to the given space being (quasi-)acted on, these 
questions are not automatically the same.

A big breakthrough was made in \cite{msw}. Theorem 1 there is as follows
(in this setting bushy is the same as having at least three ends, though
they are not the same for arbitrary bounded valence trees):
\begin{thm}
Given a cobounded quasi-action of a group $G$
on a bounded valence bushy tree $T$, there is a bounded valence, bushy
tree $T'$ and a cobounded isometric action of $G$ on $T'$ which is equivalent
to the given quasi-action of $G$ on $T$.
\end{thm}
This is a very strong conclusion and led to various important consequences
in that paper. However the hypotheses are really rather restrictive and
certainly a long way from being necessary. Why the restriction to bushy
rather than arbitrary bounded valence trees? Why not have weaker
finiteness conditions on the given tree $T$ such as being locally finite
(a quasi-action will not preserve valency in general even if it maps
vertices to vertices)? If we wish to consider an isometric action on a
bounded valence quasi-tree, we have said above that this action
is equivalent
to a quasi-action on some tree but we would need to know that this tree
can also be taken to have bounded valence in order to apply this theorem. Or
we might have a quasi-action on a bounded valence or locally finite
quasi-tree $X$. Indeed why does $X$ have to be a graph at all? We might
want to take it just to be a quasi-tree that is a proper metric space,
or even any metric space that is quasi-isometric to a tree and which also
quasi-isometrically embeds in some proper metric space. These last two
conditions are clearly necessary for the conclusion to hold, because the
existence of an equivalence will provide a quasi-isometry between the
given space $X$ and the resulting bounded valence tree $T'$.

In this paper we show that these conditions are indeed sufficient, apart
from the case of the space $X$ being quasi-isometric to $\R$. In particular
we have (as a special case of Theorem \ref{main}):
\begin{thm} \label{intr}
Given any cobounded quasi-action of a group $G$ on a metric space $X$,
where $X$ is quasi-isometric to some simplicial tree and also
quasi-isometrically embeds into a proper metric space,
exactly one of the following three cases occurs:\\
The quasi-action is equivalent to\\
$\bullet$ 
some cobounded isometric action on a bounded valence, bushy
tree\\
$\bullet$
or some cobounded quasi-action on the real line\\
$\bullet$
or to the trivial isometric action on a point.
\end{thm}
It follows that the only way we can fail to find an equivalent
action on a bounded valence tree is with the quasi-morphism example given
above. Another problem when the tree is $\R$ is that,
in the case of say $\Z\times\Z$, $\Q$ or $\R$, we might act on $\R$ by
isometries but not simplicial automorphisms which could be an issue if we
were hoping to find splittings of our group $G$ using Bass - Serre theory
(which we do obtain in the first case where the tree is bushy). 
Nevertheless we show in Corollary \ref{qttt} that if $G$ is finitely
generated and $X$ is a proper quasi-tree then
any unbounded action on $X$ implies the existence of a non trivial
splitting of $G$. To obtain these
statements, we do not vary the proof of \cite{msw} Theorem 1 but show
that the given quasi-action can be changed into an equivalent quasi-action
to which this theorem applies.

However our main aim in this paper is to deal with actions and quasi-actions
which are not cobounded. A fair bit of work has emerged recently which
examines possible actions (and sometimes quasi-actions) of a given group
on hyperbolic spaces up to equivalence. However the vast majority of these
results apply only to cobounded actions, even though non cobounded actions
are certainly very common: for instance any cobounded (quasi-)action of a
group on a space gives rise to a (quasi-)action of any of its subgroups
and these are unlikely to be cobounded themselves. But if for instance we
found ourselves with a finitely generated group $G$ having an action on a
tree $T$ which was not cobounded, we would merely restrict to the unique
minimal $G$-invariant subtree $T_G$ which certainly gives us a cobounded
action (indeed the quotient $G\backslash T_G$ is a finite graph by say
\cite{dd} Proposition I.4.13).

At first sight there is a problem with taking non cobounded actions if we
want to apply \cite{msw} Theorem 1 or even conclude that our group
  (quasi-)acts coboundedly on some well behaved space. This is that
  equivalence preserves coboundedness, so if we start with a quasi-action
  that is not cobounded, no equivalent quasi-action will be cobounded
  either. To get round this, we take an idea already in the literature
  of changing the notion of equivalence so that
  the coarse $G$-equivariant map from one space $X_1$ to another
  space $X_2$ is permitted to be a quasi-isometric embedding rather than a
  full quasi-isometry. We refer to this as a reduction from the quasi-action
  on $X_2$ to the quasi-action on $X_1$, as the idea is that $X_2$ could
  be a much bigger space than is needed to study the given quasi-action
  (for instance a reduction occurs whenever an isometric action of a group
  $G$ on a space $X$ is restricted to a $G$-invariant subspace $X_0$
  provided that $X_0$ inherits the subspace metric from $X$).

  The idea of reduction suits our purposes very well because we can
  certainly have the quasi-action on $X_1$ being cobounded but not
  the one on $X_2$ (though coboundedness will be preserved the other way
  round). Unlike equivalence, reduction is no longer an equivalence
  relation but we show (after providing some necessary metric space
  preliminaries in Section 2 and defining quasi-actions and related notions
  at the start of Section 3) in Subsection 3.3 that for any quasi-action
  there is a unique minimal reduction up to equivalence and that these
  minimal reductions are precisely the reductions to cobounded quasi-actions.
  For a cobounded quasi-action on a space $X$, the space and any quasi-orbit
  are quasi-isometric, so that an advantage of using quasi-actions is that
  they can be transferred between the two, without worrying whether an
  isometric action exists. Hence from then on our quasi-actions will be
  restricted (or perhaps ``quasi-restricted'') to a quasi-orbit. Now
  reduction does not change the quasi-isometry type of quasi-orbits,
  so if we have a quasi-action $\alpha$
  where the quasi-orbits are not quasi-isometric
  to any geodesic metric space then there will be no reduction of $\alpha$
  to any cobounded quasi-action on any geodesic metric space and so we leave
  such quasi-actions alone. Thus an important property of metric spaces
  in this paper is that of being a quasi-geodesic space $X$: this has various
  equivalent formulations but here we will simply define it as $X$ is
  quasi-isometric to some geodesic metric space. Because we are working
  with quasi-actions, we can and will move freely between geodesic and
  quasi-geodesic spaces.

  Theorem 1 of \cite{msw} does not require the group $G$ to be finitely
  generated, which is at first surprising given the conclusion. However
  the input is a cobounded quasi-action of $G$ on a suitable space and
  often infinitely generated groups will not (quasi-)act coboundedly on
  geodesic metric spaces. (For instance a metrically proper quasi-action of
  an infinitely generated group on a quasi-geodesic space is never
  cobounded.) In fact there is a similar property to being finitely
  generated but which is more general because it takes account of the
  quasi-action as well as the group. This is that quasi-orbits are coarse
  connected and it holds for all quasi-actions of finitely generated groups.
  Any quasi-geodesic space is coarse connected, thus if we have a quasi-action
  where the quasi-orbits are not coarse connected then it falls outside our
  remit and needs to be excluded from our hypotheses. Indeed strange things
  can happen, such as Subsection 10.1 Example 4 which is a (necessarily
  infinitely generated) group with an unbounded action on a locally finite
  tree but with no unbounded action on a bounded valence tree.
  
  In section 4 we use some coarse geometry to show that we can turn any
  quasi-geodesic metric space which is quasi-isometric to some proper
  metric space into a connected locally finite graph. More importantly
  if a group $G$ quasi-acts on some space $X$ and the quasi-orbits are
  quasi-geodesic metric spaces which quasi-isometrically embed in some proper
  metric space then we can reduce this quasi-action to a cobounded one
  on a locally finite graph. Moreover this graph can be taken to have
  bounded valence (because of the coboundedness).

  In Section 5 we give some background and examples of groups acting by
  isometries on hyperbolic spaces. Here a hyperbolic space will be a
  $\delta$-hyperbolic geodesic metric space
  for some appropriate $\delta$ (the
  exact value of $\delta$ will not be needed, so any equivalent definition
  can be used here). In Section 6 we restrict our hyperbolic spaces to trees
  and quasi-trees. The relevant point here is that any subset of a
  (quasi-)tree $T$ that is coarse connected is (with the subspace metric
  from $T$) a quasi-geodesic space. Thus suppose that we have a quasi-action
  of a group on any space where the quasi-orbits are
  are proper quasi-trees, or even if they quasi-isometrically embed in
  some tree and also quasi-isometrically embed in some proper metric
  space, then either the quasi-orbits are not coarse connected, whereupon
  there is no well behaved reduction of this quasi-action, or we have
  reduced it to a cobounded one on a bounded valence graph $\Gamma$.
As $\Gamma$ is quasi-isometric to a quasi-orbit, it is a quasi-tree.
We then show in Corollary \ref{lftr} (by removing a point and considering
unbounded components of the complement) that a bounded valence (respectively
locally finite) graph that is a quasi-tree is in fact quasi-isometric to a
bounded valence (respectively locally finite) tree (though many locally finite
trees are not quasi-isometric to any bounded valence tree).

Consequently in Sections 7 and 8 we can put this together with \cite{msw}
Theorem 1 to obtain the following generalisation of Theorem \ref{intr}
to actions that need not be cobounded.
\begin{thm} \label{intr2} 
Take any quasi-action with coarse connected
quasi-orbits of any group $G$ on an arbitrary metric space $X$.
Suppose that these quasi-orbits both quasi-isometrically embed
into a tree and quasi-isometrically embed into a proper metric space.
  Then exactly one of the following three cases occurs:\\
The quasi-action reduces to\\
  $\bullet$
some cobounded isometric action on a bounded valence, bushy tree\\
$\bullet$
or some cobounded quasi-action on the real line\\
$\bullet$
or to the trivial isometric action on a point.
\end{thm}  
Examples are given to demonstrate that if any of the three conditions on
the quasi-orbits fail to hold, our conclusion fails even if the other
two conditions are true. These three conditions are also clearly
necessary. 

This theorem shows us that any quasi-orbit must be quasi-isometric to a
bounded valence bushy tree, to a point, or to $\R$. In the first two cases
we now have the strongest conclusion we could have hoped for under these
very general hypotheses: a reduction to a cobounded isometric action
on a bounded valence tree. An isometric action (we do not assume
coboundedness) of an arbitrary group (not necessarily finitely
generated or even countable) acting by isometries
on an arbitrary hyperbolic space (assumed geodesic but not
proper) has two limit points if and only if the orbits of the action
are quasi-isometric to $\R$. Therefore this encompasses our remaining case
in Theorem \ref{intr2}, at least on turning the quasi-action into an
isometric action.
In Section 9 we completely determine all actions of
a given group $G$ on hyperbolic spaces with two limit points (we say the
action is orientation preserving or reversing if it respectively
fixes or swaps these two limit points). Such an action gives rise to
a Busemann quasi-morphism (homogeneous but not necessarily a homomorphism)
of $G$ at one of the limit points. (In the orientation reversing case,
we take the Busemann quasi-morphism of the index 2 orientation preserving
subgroup which is anti-symmetric with respect to the whole group.)

This consequence of a group action that fixes a point on the boundary of a
hyperbolic space has been examined by various authors. However we show in
Theorem \ref{twchr} for orientation preserving and Corollary \ref {rvrs}
for orientation reversing actions that the Busemann quasi-morphism $B$ is
not just obtained from the action, it {\it is} the action (meaning that
the natural quasi-action on $\R$ given by $B$ is the canonical reduction
of the action). This is certainly not true without the two limit point
condition. For instance the Baumslag - Solitar group
$\langle t,a\,|\,tat^{-1}=a^2\rangle$ acts both on the line and on the
Bass - Serre tree given by the splitting as an HNN extension. These are
certainly very different actions (they are not even of the same hyperbolic
type) but in both cases the Busemann quasi-morphism is the homomorphism
sending $t$ to 1 and $a$ to 0.

This all works well for actions themselves but for quasi-actions (even on
$\R$) there is no direct way to define a Busemann quasi-morphism: we have to
convert them into an equivalent isometric action and then proceed. However
the division of group elements into elliptic and loxodromic
(which makes sense for quasi-actions and results here
in no parabolic elements,
just as for actions on trees) is preserved by this process. This results
in Theorem \ref{dpell} explaining exactly when our remaining case of
a quasi-action with
quasi-orbits that are quasi-isometric to $\R$ can be reduced to a
cobounded isometric action on a bounded valence tree
(which can be assumed to be $\R$).
We summarise this as:
\begin{thm}
  Take any quasi-action of a group $G$ on a metric space where the quasi-orbits
  are quasi-isometric to $\R$. This quasi-action can
  be reduced to an isometric action on a proper hyperbolic space or on a tree
  (which we can take to be $\R$ and the action to be cobounded) if and only if
  the set of elliptic elements of the quasi-action
  is equal to the kernel
of a homomorphism from $G$ to $\R$ (or from an index 2 subgroup of $G$
in the orientation reversing case).
\end{thm}
Thus there are some quasi-actions with quasi-orbits
quasi-isometric to $\R$ which cannot be reduced to a cobounded isometric
action on any nice space, but these are the only exceptions. We finish
Section 9 by showing that if we want to insist on simplicial
rather than isometric actions on the real line in the above theorem
then we just change the target of the homomorphism from
$\R$ to $\Z$.

We end in Section 10 with applications. For a finitely generated group $G$,
an unbounded simplicial action on a tree gives rise to a splitting of the
group by Bass - Serre theory. Thus our results above tell us that an
unbounded quasi-action of $G$ on a metric space where the quasi-orbits
quasi-isometrically embed in a tree and quasi-isometrically embed in
some proper metric space gives us a splitting of $G$, except for the
case when they are quasi-isometric to $\R$. By changing the action
slightly in this case, we establish Corollary
\ref{qttt} as mentioned earlier in this introduction which says that any
unbounded isometric action of $G$ on a proper quasi-tree implies a
splitting for $G$.
Finally we give a result that requires no conditions at all
on the space and only one obviously necessary condition on the quasi-orbits. 
Theorem \ref{vfmn} states that any group with a metrically proper action
on a space where the quasi-orbits are quasi-isometric to trees is
finitely generated and virtually free. The interesting point here is that
we obtain as a corollary that a finitely generated group $G$
quasi-isometric to a tree is virtually free, rather than requiring this
fact in the proof. This major result is normally established by combining
the Stallings Theorem on splittings over finite subgroups with Dunwoody's
Accessibility Theorem. However the power of \cite{msw} Theorem 1 along
with our techniques mean that the metrically proper action of
$G$ on itself by left multiplication
gives rise to an equivalent quasi-action on the tree. This can now
be turned into a genuine action on a tree with finite stabilisers and
finite quotient graph.

\section{Metric space preliminaries}

\begin{defn} Given $c>0$, a metric space $(X,d_X)$ is said to be   
  $c$-{\bf coarse connected} if for any $x,y\in X$ we can find a
  $c$-coarse path from $x$ to $y$. This is a finite
sequence of points $x=x_0,x_1,\ldots ,x_n=y$ with $d_X(x_{i-1},x_i)\leq c$.
It is {\bf coarse connected} if it is $c$-coarse connected for some $c>0$
(in which case it is also $d$-coarse connected for every $d\geq c$).
\end{defn}
Note that any connected metric space is $c$-coarse connected for every 
$c>0$. Also for any metric space $X$ and any $c>0$, the relation: there
  exists a $c$-coarse path from $x$ to $y$ is an equivalence relation on $X$.

We introduce another similar definition.
\begin{defn} Given $C>0$, a non-empty subset $S$ of $X$ is said to be 
$C$-{\bf coarse dense} in $X$ if every point of $X$ is at distance at most 
$C$ in $X$ from some point of $S$. It is {\bf coarse dense} if it is
$C$-coarse dense for some $C>0$.
\end{defn}

Note that if $X$ is a metric space that is $c$-coarse connected then
a $C$-coarse dense subset
of $X$ is $(2C+c)$-coarse connected (with the subspace metric).

  As is well known, a $(K,\epsilon)$ {\bf quasi-isometric embedding}
  $q$ from a metric space
  $(X,d_X)$ to another $(Y,d_Y)$ has $K\geq 1,\epsilon\geq 0$ such that
  \[(1/K)d_X(x,y)-\epsilon\leq d_Y(q(x),q(y))\leq Kd_X(x,y))+\epsilon.\]
  A $(K,\epsilon,C)$ {\bf quasi-isometry} for $C\geq 0$
  is a $(K,\epsilon)$
  quasi-isometric embedding $q$ which is also $C$-coarse onto, 
  that is the image of $q$ is $C$-coarse
dense in $Y$.
  If so then there will be a quasi-inverse $r:Y\rightarrow X$ which is
  also a quasi-isometry and such that both $rq$ and $qr$ are within
  bounded distance of $Id_X$ and $Id_Y$ respectively.

Note that a quasi-isometry $q$ preserves the coarse connected 
and coarse dense properties of subsets, though the constants
can change.
Indeed if $X$ is coarse connected and $q:X\rightarrow Y$ is a quasi-isometric
embedding then $q(X)$ is also coarse connected. If moreover $q$ is
a quasi-isometry then $Y$ is coarse connected and if we have a subset $S$
which is coarse dense in $X$ then $q(S)$ is coarse dense in $Y$.

  We say that subsets $S,T$ of a metric space $X$ are {\bf Hausdorff
    close} if the Hausdorff distance between them is finite. This implies
  that $S$ and $T$ are quasi-isometric with the subspace metrics inherited
  from $X$.

Given a metric space $X$, we will sometimes want to replace
it by a graph $\Gamma$ with the same geometry as $X$.
One way of proceeding is to use the {\bf Rips graph}.
\begin{defn} \label{rps}
Given any metric space $(X,d_X)$ and any $r>0$,
the {\bf Rips graph}  $\Gamma_r(X)$ is defined to be
the graph with vertex set $X$ and an edge
between $x,y\in X$ if and only if $0<d_X(x,y)\leq r$.
\end{defn}
Note that $\Gamma_r(X)$ will be a connected graph if and only if our original
space $X$ is $r$-coarse connected. If so then we think of $\Gamma_r(X)$ as
a metric space itself with the standard path metric $d_\Gamma$
where every edge has
length 1. Indeed in this case $\Gamma_r(X)$ will be a geodesic metric
space. Moreover given any non empty subset $S$ of $X$, we can form
the Rips graph $\Gamma_r(S)$ where $S$ has the subspace metric inherited from
$X$. For instance if $X$ is itself a locally finite connected graph equipped
with the path metric (and without self loops or multiple edges)
then $\Gamma_1(X)$ is (a point or) a graph with uncountable valence at
every vertex, whereas on forming $\Gamma_1(V(X))$ for $V(X)$ the vertex set of
$X$, we get $X$ again.

However in general, even if $\Gamma_r(X)$ is a connected graph
it might not be quasi-isometric to $X$. Moreover we would like these two
spaces to be naturally quasi-isometric,
meaning that the natural map $f:X\rightarrow\Gamma_r(X)$ sending a point in
$X$ to the corresponding vertex in the Rips graph is a quasi-isometry. If so
then the map $g:\Gamma_r(X)\rightarrow X$ that first sends points on edges to
a nearest vertex and then to the corresponding point in $X$ will be a
quasi-inverse of $f$. This is true for every $r>0$ if $X$ is a geodesic
metric space. But we
will also be working with subspaces of geodesic metric spaces which
need not be geodesic metric spaces themselves. However if
we do not require this condition on the geometry of $\Gamma_r(X)$
to hold for all $r>0$ but for all sufficiently large $r$ then there
is a clear condition for exactly when this happens and it gives rise
to a wider class of metric spaces than the geodesic metric spaces.
\begin{thm} \label{qugd}
  Let $(X,d_X)$ be a metric space and $\Gamma_r(X)$ be the Rips graph for any
  $r>0$. Then $\Gamma_r(X)$ is connected if and only if $X$ is $r$-coarse
  connected, in which case $\Gamma_s(X)$ is connected for all $s\geq r$.\\
  If $X$ is coarse connected then the following are equivalent:\\
(i) There exists $s$ with $\Gamma_s(X)$ connected and quasi-isometric to $X.\\$
(ii) $X$ is quasi-isometric to some geodesic metric space $(Y,d_Y)$.\\
(iii) For all sufficiently large $s$, the natural inclusion of $X$ in
$\Gamma_s(X)$ is a quasi-isometry.
\end{thm}
\begin{proof}
The first part is clear from the definition, whereupon we will have
  \[d_X(a,b)\leq r d_\Gamma(a,b)\]
  for all $a,b\in X$. Again (i) implies (ii) is clear since
  $\Gamma_s(X)$ with the path metric is a geodesic metric space. For (ii)
  implies (iii), we invoke \cite{dcdlh} Chapter 3. In particular 
  Definition 3.B.1 defines a space $X$ being ($c$-)large scale geodesic which
  is seen to be preserved (though not the value of $c$) under quasi-isometry.
  This large scale geodesic property is satisfied by all geodesic spaces
  for any $c>0$ and for a general space $X$ it
  is equivalent to saying that the natural map from $X$ to $\Gamma_c(X)$ is a
  quasi-isometry. Thus if $X$ is quasi-isometric to a geodesic space $Y$
  then $X$ is $s$-large scale geodesic for all large $s$. Now (iii)
  implies (i) is clear.
\end{proof}

Because of this theorem, we can define the following.
\begin{defn}
  A metric space $(X,d_X)$ is a {\bf quasi-geodesic space} if it is
  quasi-isometric to some geodesic metric space.
\end{defn}
There are various equivalent definitions of quasi-geodesic spaces in the
literature, sometimes using this name and sometimes using other terminology.

\section{Group actions and quasi-actions}

\subsection{Definitions and Properties}

So far we have only dealt with metric spaces but our
interest here is with isometric actions of a group on a metric space $X$
or indeed the more general notion of a quasi-action, which we now define.

\begin{defn}
Given $K\geq 1$ and $\epsilon,C\geq 0$,  
a $(K,\epsilon,C)$ {\bf quasi-action} of a group $G$ on
a metric space $(X,d_X)$ is a map $\alpha:G\times X\rightarrow X$ such that:\\
(1)
For each $g\in G$ the map $A_g$ given by $A_g(x)=\alpha(g,x)$ is a
$(K,\epsilon,C)$ quasi-isometry of $X$.\\
(2) For every $x\in X$ we have $d_X(\alpha(id,x),x)\leq C$.\\
(3) For every $x\in X$ and $g,h\in G$ we have
\[d_X(\alpha(g,\alpha(h,x)),\alpha(gh,x))\leq C.\]

A {\bf quasi-action} is a $(K,\epsilon,C)$ quasi-action for some
$K,\epsilon,C$.
\end{defn}
Note: there is variation in the literature on whether (2) is included but
it actually follows from the other axioms (because even
without (2) the map
$g\mapsto [A_g]$ from $G$ to the quasi-isometry group of $X$ is a
homomorphism). Therefore we will feel free to omit (2) in proofs.

Obviously a $(1,0,0)$-quasi-action is the same as
an isometric action. Note that
if $C=0$ then we have an action under the usual definition of $G$ on $X$
considered as a set and in particular all elements of $G$ will give rise to
bijections. However this will be an action by (uniform)
quasi-isometries rather than isometries.

The following is the main source of examples for quasi-actions which
are not genuine isometric actions.
\begin{ex} \label{qunew}
Suppose we have two metric spaces $X$ and $Y$
and a quasi-isometry $q:X\rightarrow Y$ with quasi-inverse $r:Y\rightarrow X$.
If $\alpha$ is a quasi-action of $G$ on $X$ (and in particular if it is
a genuine isometric action) then it is easily checked that
$\beta(g,y):=q(\alpha(g,r(y)))$ is a quasi-action of $G$ on $Y$.
\end{ex}

We now define the main concepts obtained from a quasi-action $\alpha$
which we will
be using. These all reduce to well known definitions if $\alpha$ is a genuine
action but sometimes the behaviour is a little different in the more general
case.
\begin{defn} If $\alpha$ is a quasi-action of $G$ on $X$ then for 
$x_0\in X$ the
  {\bf quasi-orbit} ${\mathcal Q}(x_0)$ of $x_0$ is the subset
    $\{\alpha(g,x_0)\,|\,g\in G\}$ of $X$.
  \end{defn}
  Clearly this reduces to the standard notion of an orbit for a genuine action
  but note that in general ${\mathcal Q}(x_0)$ need not be closed under
  composition (or even contain $x_0$):
  $A_g(A_h(x_0))=\alpha(g,\alpha(h,x_0))$ need not be in ${\mathcal Q}(x_0)$,
  although it will be near the point $\alpha(gh,x_0)$ which is.
\begin{defn} We say the quasi-action $\alpha$ of $G$ on $X$ is {\bf bounded}
(which is shorthand for has bounded quasi-orbits)
  if there exists $x_0\in X$ such that the quasi-orbit ${\mathcal Q}(x_0)$
    is a bounded subset of $X$.
\end{defn}
Note that if one quasi-orbit is bounded then they all are, because
quasi-orbits are quasi-isometric (any two are Hausdorff close).
We now look at the case of quasi-actions of cyclic groups, following
\cite{mnlms} Definition 2.13.
\begin{defn} \label{elp}
  Given a quasi-action $\alpha$ of a group $G$ on $X$ and an
  element $g\in G$, we say that:\\
  $g$ is {\bf elliptic} if the quasi-action $\alpha$ restricted to
  $\langle g\rangle$ is bounded; that is\\
  $\{\alpha(g^n,x_0)\,|\, n\in\Z\}$
  is a bounded set.\\
  $g$ is {\bf loxodromic} if $g$ 
  the quasi-action $\alpha$ restricted to $\langle g\rangle$
  quasi-isometrically
  embeds; that is the map $n\mapsto \alpha(g^n,x_0)$ is a quasi-isometric
  embedding of $\Z$ in $X$.\\
  $g$ is {\bf parabolic} if it is neither elliptic nor loxodromic.
\end{defn}
It is again easily checked that this definition is independent of the point
$x_0$, that $g$ and $g^k$ have the same type for $k\neq 0$ and that conjugate
elements have the same type.
\begin{defn} \label{dfcb}
We say that the $(K,\epsilon,C)$ quasi-action $\alpha$ of $G$
    on $X$ is
{\bf cobounded} if there is a point $x_0\in X$ whose quasi-orbit
${\mathcal Q}(x_0)$ is a coarse dense subset of $X$.
\end{defn}
Note that this definition is (uniformly) independent of $x_0$ (if
${\mathcal Q}(x_0)$ is $c$-coarse dense and we take any $x\in X$ then
${\mathcal Q}(x)$ is $(c(K+1)+\epsilon+C)$-coarse dense).
\begin{defn} \label{dfmp}
We say that the $(K,\epsilon,C)$
  quasi-action $\alpha$ of $G$ on $X$
is {\bf metrically proper} if there is $x_0\in X$ such that for all $R\geq 0$
  the set $\{g\in G\,|\,d_X(\alpha(g,x_0),x_0)\leq R\}$ is finite.
\end{defn}
Again the choice of $x_0$ does not matter as for any $x\in X$ and $R\geq 0$
we have
\[\{g\in G\,|\,d_X(\alpha(g,x),x)\leq R\}\subseteq
  \{g\in G\,|\,d_X(\alpha(g,x_0),x_0)\leq R+(K+1)c_x+\epsilon\}\]
where $c_x=d_X(x,x_0)$.

There are various notions of a proper action in the literature. These are
not equivalent, even for isometric actions. Indeed acting metrically
properly is stronger in general for isometric actions than
saying that $\{g\in G\,|\,g(K)\cap K\neq\emptyset\}$ is finite for all
compact subsets $K$ of $X$. On the other hand our definition is weaker than
the definition of a proper quasi-action given in \cite{msw} which would
usually be called uniformly
metrically proper. However in this paper, the only reference that will
be made about quasi-actions possessing some sort of properness 
will be Definition \ref{dfmp}.

\subsection{Comparing actions and quasi-actions}
Suppose a given group $G$ has (quasi-)actions on two metric spaces $X,Y$.
We need a notion of when these (quasi-)actions are similar enough to be
regarded as essentially equivalent. We would hope that any of the above
properties of the (quasi-)action would be preserved in this case. If we
first regard $X$ and $Y$ as just sets then the relevant notion is that
of a map $F:X\rightarrow Y$ being equivariant, namely $F(g(x))=g(F(x))$
for all $g\in G$ and $x\in X$ or $F(\alpha(g,x))=\beta(g,F(x))$ in the
notation of quasi-actions $\alpha$ on $X$ and $\beta$ on $Y$.
The following coarse version is widely used.

\begin{defn}
  Suppose that a group $G$ has quasi-actions $\alpha$ and $\beta$ on
  the metric spaces $(X,d_X)$ and $(Y,d_Y)$ respectively.
We say that a function $F:X\rightarrow Y$ is {\bf coarse $G$-equivariant}
if there is $M\geq 0$ such that for all $x\in X$ and all $g\in G$ we
have $d_Y(F(\alpha(g,x)),\beta(g,F(x)))\leq M$.
\end{defn}

\begin{ex} \label{equvt}
  (1) Suppose that $G$ quasi-acts on an arbitrary metric space $X$ via $\alpha$
  and $Y$ is a bounded metric space. Then any function
  $\beta:G\times Y\rightarrow Y$ is a quasi-action of $G$ on
  $Y$ and any map $F:X\rightarrow Y$ is a coarse $G$-equivariant map.\\
  (2) Now suppose that $X$ is any connected graph on which $G$ acts by
  automorphisms, thus also by isometries and $Y$ is the vertex set
  of $X$ with the same path metric. Then $G$ also acts on $Y$ by restriction.
We can create a (non-canonical) map $F$ which sends any
point of $X$ to a nearest vertex $v\in Y$ and this is a coarse $G$-equivariant
map with $M=1$.  
\end{ex}

However coarse $G$-equivariant maps on their own are not so useful without
some sort of restriction on the map itself (they need not even be preserved
under composition).
We would like $G$-equivariant maps $F$ which preserve
(perhaps forwards or backwards but ideally both ways)
properties of a quasi-action. For this we can take
$F$ to be a quasi-isometry and indeed we have even seen how to use such maps
to create new quasi-actions in Example \ref{qunew}. For quasi-actions of
a fixed group $G$ on arbitrary metric spaces, the notion of a
{\bf coarse $G$-equivariant quasi-isometry} is an equivalence relation
(for instance see \cite{for} Lemma 7.2 for isometric actions, which goes
through straightforwardly when generalised to quasi-actions).
If a group $G$ has quasi-actions $\alpha,\beta$ on spaces $X,Y$
respectively and there is a coarse $G$-equivariant quasi-isometry
from one space to the other, we will sometimes abbreviate this by saying
that the two quasi-actions of $G$ are {\bf equivalent} (this is also
referred to as being quasi-conjugate).

However this relation between two actions is sometimes too rigid for our
purposes, so we consider a more flexible version.
\begin{defn}
  If a group $G$ has quasi-actions $\alpha,\beta$ on the spaces $X,Y$
  respectively then a {\bf coarse $G$-equivariant quasi-isometric
    embedding} from $\alpha$ to $\beta$ is a quasi-isometric embedding
  from $X$ to $Y$ which is coarse $G$-equivariant with respect to
  $\alpha$ on $X$ and $\beta$ on $Y$. We say that the quasi-action $\beta$
  {\bf reduces} to $\alpha$.
\end{defn}  
This is not an equivalence relation as the examples below show.
However it is easily checked that
coarse $G$-equivariant quasi-isometric embeddings compose as such, so
reduction of quasi-actions is transitive.
\begin{ex} \label{ex2}
  (1) In (1) from Example \ref{equvt}, $F$ will not be a quasi-isometry or
  quasi-isometric embedding unless $X$ is bounded (in which case it is
  both). Now suppose that $X$ is unbounded and for $x_0\in X$
  we consider the constant map $F:Y\rightarrow X$ going the other way, given
  by $F(y)=x_0$. Then $F$ clearly is not a quasi-isometry. It is a
  coarse $G$-equivariant quasi-isometric embedding if and only if the
  quasi-action of $G$ on $X$ is bounded.\\
  (2) In (2) from Example \ref{equvt} $F$ is a coarse $G$-equivariant
  quasi-isometry and for a (coarse) $G$-equivariant quasi-inverse of $F$
  we can take the map from $Y$ to $X$ which sends a vertex to itself.\\
(3) In Example \ref{qunew},  $q$ and $r$ both become coarse $G$-equivariant
quasi-isometries. A particularly useful case of this is where $Y$ is a
connected graph $\Gamma$ equipped with the path metric and $X$ is the vertex
set of $\Gamma$, with $q$ the vertex inclusion map and $r$ a nearest vertex
map $F$ as in Example \ref{equvt} (2). Thus if we have some quasi-action
$\alpha$ defined on $X$ then we can immediately assume
that $\alpha$ extends to a quasi-action on all of $\Gamma$. As we have a
coarse $G$-equivariant quasi-isometry between the original action and its
extension, we can in effect regard them as the same quasi-action and we
will use this in what follows. Similarly an arbitrary quasi-action
of a group on $\Gamma$ can be assumed to send vertices to vertices.
Or we can even replace the graph $\Gamma$ with the vertex set $X$ as the
space on which $G$ quasi-acts.\\
(4) Given a quasi-action $\alpha$ of $G$ on $X$ with basepoint $x_0$, take any
element $g\in G$. Then the map $F:\Z\rightarrow
X$ given by $F(n)=\alpha(g^n,x_0)$ is a coarse $\langle g\rangle$-equivariant
quasi-isometric embedding if and only if $g$ is loxodromic. If so then we
can also regard this as a coarse $\langle g\rangle$-equivariant
quasi-isometric embedding from $\R$ to $X$ by the point above in (3).
However $F$ will not be a quasi-isometry (unless $X$ itself is quasi-isometric
to $\R$). 
\end{ex}

It is straightforward to see
using the various definitions that if $\alpha$ on $X$
and $\beta$ on $Y$ are two quasi-actions of a group $G$ which are related
by a coarse $G$-equivariant quasi-isometry $F$ then any of the above
properties of quasi-actions are preserved by $F$. We record without
proof the equivalent statements for quasi-isometric embeddings, which all
follow by the usual arguments involving manipulating definitions and
playing with constants.
\begin{prop} \label{qpres}
  Let $F:X\rightarrow Y$ be a coarse $G$-equivariant quasi-isometric
  embedding with respect to the quasi-actions $\alpha$ of $G$ on $X$ and
  $\beta$ of $G$ on $Y$. Then:\\
  (1) $\alpha$ is bounded if and only if $\beta$ is bounded.\\
  (2) Any element $g\in G$ quasi-acts elliptically (respectively
  loxodromically, respectively parabolically) under $\alpha$ if and only if
  it quasi-acts
  elliptically (respectively loxodromically, respectively parabolically)
  under $\beta$.\\
  (3) The quasi-action $\alpha$ is metrically proper if and only if the
  quasi-action $\beta$ is metrically proper.\\
  (4) If the quasi-action $\beta$ is cobounded then the quasi-action
  $\alpha$ is also cobounded but not necessarily the other way around.
\end{prop}  

Counterexamples in (4) are easily given by embedding $X$ in a much bigger
space $Y$. This failure is crucial for us because it means that given a
non-cobounded quasi-action $\beta$ on a space $Y$,
we will be able to find a cobounded quasi-action $\alpha$ of $G$ on
a space $X$ and a coarse $G$-equivariant quasi-isometric
embedding $F$ from $X$ to $Y$ (and we will give conditions on when we can
take $X$ to be a suitably well behaved space).
This means that all relevant properties
of $\beta$ (apart from non-coboundedness of course) will be preserved by
$\alpha$ but the latter quasi-action should be easier to study.
If we insisted further that $F$ was a quasi-isometry then (by using
Proposition \ref{qpres} (4) and taking a quasi-inverse) $\alpha$ is
cobounded if and only if $\beta$ is.
Thus with this approach, examining cobounded actions would not tell
us about non-cobounded ones.
  
The use of quasi-actions above
begs the following question. Suppose that $G$ has a
quasi-action $\alpha$ on a metric space $X$. As actions are easier to
understand than quasi-actions, is there a genuine isometric action $\beta$
of $G$ on $X$ itself, or on some other space $Y$ quasi-isometric
to $X$, which is equivalent to $\alpha$?

If we insist on keeping
the same space $X$ then the answer is definitely no. For instance we could
have some space $X'$ quasi-isometric to $X$ where $X'$ admits an unbounded
isometric action $\alpha$ but the isometry group of $X$ is trivial. Then
we can use $\alpha$ to form an unbounded quasi-action on $X$ via Example
\ref{qunew} but there are no such actions. However the answer
is yes if we allow ourselves to change the space, as shown by the
following results of Kleiner and Leeb, and of Manning. 
\begin{thm} \label{gnact}
  (i) (\cite{kllb}) Given any quasi-action $\alpha$ of any group $G$ on a metric
  space $X$,
there is an isometric action $\beta$ of $G$ on a metric space $Y$ and a coarse
$G$-equivariant quasi-isometry from $\alpha$ to $\beta$.\\
(ii) (\cite{mngeom} Proposition 4.4)
Given any quasi-action $\alpha$ of any group $G$ on a
geodesic metric space $X$,
there is an isometric action $\gamma$ of $G$ on a geodesic
metric space $Z$ and a coarse
$G$-equivariant quasi-isometry from $\alpha$ to $\gamma$.
\end{thm}
In Theorem \ref{gnact} (ii) we can obviously replace $X$ being geodesic
with $X$ being quasi-geodesic by Example \ref{qunew}.

\subsection{Reductions of actions and quasi-actions}

Given a quasi-action $\alpha$ of $G$ on a metric space $X$,
we now consider reductions of $\alpha$ to other quasi-actions
and in particular whether there is (are)
a preferred reduction(s). There are two features we might ask for:
a reduction where the underlying
space $Y$ is well behaved, say a geodesic space, or a reduction that
is minimal. In each case, we would ideally want to end up with an isometric
action rather than just a quasi-action.

To deal with the first question,
suppose initially that we have a group $G$ with a genuine isometric action
on some metric space $(X,d_X)$ which we would like to reduce to an
action on some geodesic space if possible (here we do not assume that $X$
is itself a quasi-geodesic space).
Let us return to Definition \ref{rps} where we considered Rips
graphs. On forming $\Gamma_r(X)$ for a given $r>0$, we can define
$G$ to act on the vertices of $\Gamma_r(X)$ in the same
way, at least combinatorially, as it does on $X$.
Moreover this clearly extends to an
action of $G$ on $\Gamma_r(X)$ by graph automorphisms, because for
$x,y\in X$ the question of whether $d_X(x,y)\leq r$ is the same as whether
$d_X(gx,gy)\leq r$. 
Therefore if $r$ is such that $\Gamma_r(X)$ is
connected, we have a natural action of $G$ on $\Gamma_r(X)$ by
isometries. Moreover this construction also works on replacing $X$ with
any subset $S$ of $X$ (using the subspace metric for $S$)
which is $G$-invariant and again if we have $s>0$ with $\Gamma_s(S)$ connected
then we have an isometric action of $G$ on $\Gamma_s(S)$ too. If further
we have that $S$ is a quasi-geodesic metric space then the conditions
of Theorem \ref{qugd} are satisfied. This allows us to
conclude that if we can find a $G$-invariant
subset of $X$ which is a quasi-geodesic space
then we obtain a coarse $G$-equivariant quasi-isometry from
$\Gamma_s(S)$ to $S$ and thence a reduction of our action on $X$
to an action on the geodesic space $\Gamma_s(S)$.

Now suppose that $G$ has a quasi-action on our space $X$, rather than
an isometric action. Theorem \ref{gnact} (i) means that we can first replace
$\alpha$ with an equivalent isometric action $\beta$ on another metric space
$Y$, then use our comment above. Thus we are after subsets of $Y$ which
are invariant under this isometric action, with the obvious choice 
of restricting $\beta$ to a single orbit
$G\cdot y_0$ (with the subspace metric inherited from $Y$). This means
we have reduced $\alpha$ to an isometric action which is now cobounded.
The problem with this approach
is that we could end up with a badly behaved space: for instance if $G$ is
a countable group then $G\cdot y_0$ will never be a geodesic space (unless
it is a single point). However if
$G\cdot y_0$ were a quasi-geodesic space then 
we could now apply Theorem \ref{gnact} (ii) to this restriction
of $\beta$ thus
obtaining a reduction of $\beta$, and hence $\alpha$, to a cobounded
isometric
action on a geodesic space. It turns out that this condition on a
(quasi-)orbit is necessary, because if the
original quasi-orbits are badly behaved then this behaviour will persist:
\begin{prop} \label{orgi} If we have a reduction of the quasi-action $\alpha$
  on some space $X$ to the quasi-action $\beta$ on the space $Y$ then
  any quasi-orbit ${\mathcal Q}(x_0)\subseteq X$
  of $\alpha$ is quasi-isometric to any
  quasi-orbit ${\mathcal Q}(y_0)\subseteq Y$ of $\beta$.
\end{prop}
\begin{proof}
If $F$ is the coarse $G$-equivariant quasi-isometric embedding from    
$Y$ to $X$ then we need only check this for one quasi-orbit from
each space. Thus on taking $y_0\in Y$, we just consider
${\mathcal Q}(y_0)$ and ${\mathcal Q}(F(y_0))$. Now the former set
is quasi-isometric to to $F{\mathcal Q}(y_0)$ because $F$ is a 
quasi-isometric embedding which is onto when restricted to
${\mathcal Q}(y_0)$. Although $F{\mathcal Q}(y_0)$ need not be equal to
${\mathcal Q}(F(y_0))$, they are
Hausdorff close as $F$ is coarse $G$-equivariant, hence quasi-isometric.
\end{proof}

This gives us the following lemma.
\begin{lem} \label{ccnq}
  Suppose that $\alpha$ is a quasi-action of the group $G$ on
  the metric space $X$ which reduces to the quasi-action $\beta$ on the
  space $Y$.\\
  (i) Quasi-orbits of $\alpha$ are coarse connected (respectively
  quasi-geodesic spaces) if and only if quasi-orbits of $\beta$ are
  coarse connected (respectively quasi-geodesic spaces).\\
  (ii) If the quasi-orbits of $\alpha$ are not quasi-geodesic spaces then
  $\alpha$ cannot be reduced to a cobounded quasi-action on any quasi-geodesic
  space.\\
  (iii) If the quasi-orbits of $\alpha$ are quasi-geodesic spaces then there
  exists a reduction of $\alpha$ to a cobounded isometric action on a
  geodesic metric space.
\end{lem}
\begin{proof}
Part (i) follows immediately from Proposition \ref{orgi}. For part (ii),
if $\beta$ is a cobounded quasi-action on the quasi-geodesic space 
$Y$ that is a reduction of $\alpha$ then any quasi-orbit of $\beta$ is
quasi-isometric to $Y$, but then any quasi-orbit of $\alpha$ would be too,
which is a contradiction
by part (i). For (iii), assume without loss of generality that $\alpha$ is an
isometric action by Theorem \ref{gnact} (i). As any orbit $G\cdot x_0$ is
a quasi-geodesic space by (i), we can take the Rips orbit graph
$\Gamma_s(G\cdot x_0)$ and argue exactly as was given at the start
of this subsection. The action of $G$ on this connected graph is by
isometries and is vertex transitive, hence cobounded.
\end{proof}

We finish this subsection on reductions by showing that for a given
quasi-action $\alpha$, it is the reductions of $\alpha$  to cobounded
quasi-actions  that are exactly the minimal reductions in a precise sense.
We first note that as a relation, reduction is not quite as
well behaved as equivalence for
quasi-actions because it is possible to have quasi-actions $\alpha,\beta$
where $\alpha$ reduces to $\beta$ and $\beta$ to $\alpha$ but they are
not equivalent (for instance the identity acting on two spaces $X$ and $Y$
which are not quasi-isometric but where $X$ quasi-isometrically embeds
in $Y$ and vice versa). However we do have:
\begin{prop} \label{mnrd}
  Let $\alpha$ be any quasi-action on the space $X$ and suppose that
  $\alpha$ reduces to the quasi-action $\beta$ on $Y$ and $\gamma$
  on $Z$.\\
  (i) If the quasi-action $\beta$ is cobounded then $\gamma$ reduces
  to $\beta$.\\
  (ii) If further the quasi-action $\gamma$ is cobounded then $\beta$
  and $\gamma$ are equivalent quasi-actions.
\end{prop}
\begin{proof}
  By Theorem \ref{gnact} (i), whenever we have a quasi-action we can replace
  it with an equivalent action, so we can assume throughout the proof
  that we are dealing with isometric actions (and use the standard action
  notation). This allows us to restrict actions to an orbit.
  We first require a lemma that is a coarse equivariant version of
  Proposition \ref{orgi}.
\begin{lem} \label{orbq}
  (i) Given an isometric action of $G$ on some space $X$ and
  $x_1,x_2\in X$, there is a coarse $G$-equivariant quasi-isometry $\pi$ from
$Orb(x_1)$ to $Orb(x_2)$ (where both are equipped with the subspace metrics).\\
  (ii) Suppose that this action on $X$ reduces to
some action on a space $Y$ via the coarse $G$-equivariant
  quasi-isometric embedding $F:Y\rightarrow X$.
  Then for any $y_0\in Y$ and $x_0\in X$, there is a coarse $G$-equivariant
  quasi-isometry $\overline{F}:Orb(y_0)\rightarrow Orb(x_0)$.
\end{lem}
\begin{proof}
  We give the definitions of the maps used and leave the straightforward
  details to be checked. In (i) we are in a situation, which will again
  occur, that we would obviously like to take $g\in G$
  and let $\pi$ send $gx_1$ to $gx_2$ but this might not be
  well defined as we could have $gx_1=g'x_1$ but
  $gx_2\neq g'x_2$. In these situations there are various fixes: one is to
  replace the action of $G$ on the space $X$ with
  the action on $X\times G$ given by letting $g$ send $(x,h)$ to
  $(gx,gh)$. This is now a free action and we can put the discrete metric
  on $G$ (or even turn it into a complete graph).

A more ad hoc approach, which we will adopt here because of its flexibility,
is to try and stick to the above formula but to make an arbitrary
  choice when required. Thus we define $\pi$ as mapping $gx_1$ to $h_gx_2$
  where $h_g$ is some fixed choice amongst the elements
  $\{h\in G\,|\,hx_1=gx_1\}$.
Then $\pi$ is coarse dense because on taking a point $gx_2$ in $Orb(x_2)$
we have   $d_X(\pi(gx_1),gx_2)=d_X(h_gx_2,gx_2)$
  which is at most
  \[  d_X(h_gx_2,h_gx_1)+d_X(h_gx_1,gx_1)+d_X(gx_1,gx_2)\]
  and this is at most $2d_X(x_1,x_2)$ as $h_gx_1=gx_1$. Now $\pi$ being
  $G$-equivariant and a quasi-isometric embedding are checked similarly.

  For (ii), first set $x_0$ to be $Fy_0$. Although  
  $F$ itself does not map $Orb(y_0)$ to $Orb(Fy_0)$,
  being coarse $G$-equivariant means that we have $d_X(F(gy_0),gF(y_0))$
  $\leq M$.
  Thus on running the same trick as before to get a well defined map,
  we obtain $\overline{F}:Orb(y_0)\rightarrow Orb(Fy_0)$ which is defined
  as $\overline{F}(gy_0)=h_g(Fy_0)$ where $h_g$ is some element of $G$ with
  $h_gy_0=gy_0$. Then similar estimates as for (i) show that $\overline{F}$
  is coarse onto, a quasi-isometric embedding and coarse $G$-equivariant.

  That was for $x_0=Fy_0$ but now we can apply Part (i) to go from
  $Orb(Fy_0)$ to $Orb(x_0)$.
  \end{proof}
As for the proof of Proposition \ref{mnrd} (i), for any $y_0\in Y$ the 
orbit $Orb(y_0)$ is dense in $Y$ so inclusion of $Orb(y_0)$ in $Y$
is a $G$-equivariant isometric embedding from the action of $G$
on $Orb(y_0)$ to the action of $G$ on $Y$ with coarse dense image.
Thus these actions are equivalent
and so we now need a coarse $G$-equivariant quasi-isometric embedding of
$Orb(y_0)$ in $Z$. 
But by Lemma \ref{orbq} (ii), on taking any $x_0\in X$ and $z_0\in Z$
we have coarse $G$-equivariant
quasi-isometries from $Orb(y_0)$ to $Orb(x_0)$ and from $Orb(z_0)$ to
to $Orb(x_0)$, thus also from $Orb(x_0)$ to $Orb(z_0)$ and thence
from $Orb(y_0)$ to $Orb(z_0)$ by composition, with $Orb(z_0)$ lying in $Z$.

For (ii), as $\gamma$ is now cobounded we can replace $Z$ with
$Orb(z_0)$, just as we did for $Y$ and $Orb(y_0)$. Now
part (i) tells us that we have a coarse $G$-equivariant quasi-isometry from
$Y$ to $Orb(z_0)$, and thus one also from $Y$ to $Z$ because $Orb(z_0)$
is coarse dense in $Z$.
\end{proof}

\subsection{Good and bad orbits}

The previous subsection told us that given any quasi-action on an arbitrary
metric space $X$,
we have a canonical reduction of this quasi-action up to equivalence
and that it is exactly the reductions to cobounded quasi-actions which
make up this equivalence class. This means that it is the geometric
behaviour of a quasi-orbit that will determine whether these reductions
result in well behaved spaces, as opposed to the geometric behaviour of
the original space $X$. Our hope is that 
the quasi-orbits of a quasi-action are quasi-geodesic
spaces, whereupon we have seen that we can reduce the quasi-action
to a cobounded isometric action on a geodesic space.
As being a quasi-geodesic space implies coarse connectedness, we
introduce the following definition.
\begin{defn} \label{gdbd}
A quasi-action of a group on a metric space has
{\bf good quasi-orbits} if one (equivalently all) quasi-orbit(s) is (are)
quasi-geodesic spaces. It has
{\bf bad quasi-orbits} if one (equivalently all) quasi-orbit(s) is (are) not
  coarse connected.
\end{defn}
Examples of isometric actions on hyperbolic spaces whose quasi-orbits are
neither good nor bad are given in Section 5. However we will then see that
for spaces isometric to a simplicial tree, the quasi-orbits of any
quasi-action are either good or bad.

We finish this subsection with two results that are known for isometric
actions (for instance, see \cite{dcdlh}) and where the standard proofs
immediately extend to quasi-actions by taking care of the extra constants.
\begin{prop} \label{fgn}
  Let an arbitrary group $G$ have some $(K,\epsilon,C)$
  quasi-action $\alpha$
  on any metric space $(X,d_X)$.\\
(i) If $G$ is finitely generated then any quasi-orbit
${\mathcal Q}(x_0)$ is $r$-coarse connected with the subspace metric
from $X$. (Here $r$ depends on $x_0$ and the finite generating set used.)\\
(ii) Now suppose that the quasi-action $\alpha$ is metrically proper.
Then the quasi-orbits are coarse connected if and only if $G$ is finitely
generated.
\end{prop}
{\bf Example}: The free group $F(x_1,x_2,\ldots)$ of countably
infinite rank acts on its own Cayley graph with respect
to the generating set $x_1,x_2,\ldots $, which is a locally infinite tree
and the orbits are coarse connected, indeed are quasi-geodesic spaces.
But this group is a subgroup of
$F_2$, so it also acts metrically properly on the regular tree of degree 4
with orbits that are not coarse connected, by Proposition \ref{fgn} (ii).

\section{Quasi-Proper metric spaces}

We have seen how we might go about reducing a quasi-action $\alpha$ of
$G$ on $X$
to something which is better behaved by using Rips graphs
but we have not yet said anything about whether this preserves any finiteness
properties of $X$. For instance if $X$ is a proper metric space which is
also geodesic or quasi-geodesic then we
might hope that we can find $S\subseteq X$ and $s>0$ such that $\Gamma_s(S)$
is a locally finite graph with a quasi-action of $G$ which is
equivalent to (or a reduction of)
$\alpha$. This will not work if we proceed naively, for
instance if we take $S=X$ for $X$ a geodesic space then $\Gamma_s(X)$
will have all vertices with uncountable valence. However if we choose
carefully then this works. We first introduce
some definitions and results from \cite{dcdlh} Section 3.D. Here and
throughout $B(x,r)$
denotes the closed ball of radius $r\geq 0$ about $x$ in a metric space.

\begin{defn} 
A metric space $X$ is {\bf coarsely proper} if there exists $R_0\geq 0$
such that every ball $B(x,r)$ in $X$ can be covered by finitely many
balls of radius $R_0$.\\
A metric space $X$ is {\bf uniformly coarsely proper} if there exists
$R_0\geq 0$ such that for every $R\geq 0$, there exists an integer $N_r$ where
every ball $B(x,r)$ can be covered by at most $N_r$ balls of radius $R_0$.
\end{defn}
Proposition 3.D.11 (respectively Proposition 3.D.16) of \cite{dcdlh} shows
that being coarsely proper (respectively uniformly coarsely proper) 
is a quasi-isometry invariant for metric
spaces. It also shows that if $X$ is a quasi-geodesic space then this
condition is the same as being quasi-isometric to a connected graph
that is locally finite (respectively of bounded valence). Moreover
\cite{dcdlh} Proposition 3.D.13 (1) states that any proper metric
space is coarsely proper. (However Example \ref{root} gives many examples of
locally finite simplicial trees, hence proper geodesic metric spaces,
which are not quasi-isometric to any bounded valence graph, hence are
not uniformly coarsely proper.)

\begin{thm} \label{prpr}
  If $X$ is a quasi-geodesic metric space that is proper, or even
  quasi-isometric to some proper metric space, then there exists a coarse
dense countable subset $S$ of $X$
and $s>0$ such that the Rips graph $\Gamma_s(S)$ is
connected, is locally finite and the map from $\Gamma_s(S)$ to $X$ given
by sending any point to a nearest vertex and then to that point in
$S\subseteq X$ is a quasi-isometry.
\end{thm}
\begin{proof}
  If $X$ is not proper then we can replace it by a new space
  (quasi-isometric to $X$ and here also called $X$) which is proper and 
which will still be quasi-geodesic. The results quoted above then tell us
that $X$ is also coarsely proper and so is quasi-isometric to a locally
finite connected graph. (The proof proceeds by
taking an $\epsilon$-net of $X$,
  that is a subset $S$ where $d_X(a,b)\geq\epsilon$ for all distinct
  $a,b\in S$ and applying Zorn's lemma: see also \cite{bh}
  Proposition I.8.45.)
\end{proof}  
Note that these two conditions on $X$ are necessary for the conclusion to
hold since the locally finite connected graph $\Gamma_s(S)$ will be a proper
geodesic metric space quasi-isometric to $X$. However as mentioned above,
we cannot replace locally finite with bounded valence in the conclusion.

But what happens in this theorem if 
there is a quasi-action or isometric action on $X$?
In the case of isometric
actions on a metric space $X$, on taking $S$ as in the statement of
Theorem \ref{prpr} we can only really get a natural isometric action on
$\Gamma_s(S)$ if $S$ is $G$-invariant. But this might well not be the
case, for instance if $G$ is acting with all orbits indiscrete on $X$ then
any attempt to find a $\epsilon$-net fails as soon as it contains
one $G$-orbit. However we do not have this problem for quasi-actions,
because we are able to apply Example \ref{qunew}. Moreover, as we saw
in the previous section, there is much to be gained from considering what
happens to quasi-orbits ${\mathcal Q}(x_0)$ rather than the whole space
$X$. Thus we might wonder whether $X$ or ${\mathcal Q}(x_0)$ is quasi-isometric
to some proper metric space. To this end we have:
\begin{lem} \label{qiepr}
A metric space $Z$ that quasi-isometrically embeds in some proper metric space
$P$ is itself quasi-isometric to a proper metric space.
\end{lem}
\begin{proof}
  Our space $Z$ is of course quasi-isometric to its image $S$ (with the
  subspace metric) in $P$ under this embedding. Now a closed subspace of a
  proper metric space is also proper and $S$ is coarse dense in its
  closure $\overline{S}$ in $P$, thus $S$ is quasi-isometric to the proper
  space $\overline{S}$.
\end{proof}
This means that if our space $X$ is proper, or merely if it quasi-isometrically
embeds in a proper metric space then any non empty subset of $X$, and in
particular any quasi-orbit ${\mathcal Q}(x_0)$
of any quasi-action on $X$ is quasi-isometric to some proper metric space.
On the other hand we could easily have proper (quasi-)orbits inside spaces
that are much to big to be proper, or even to be quasi-isometric to something
that is proper, so we now proceed using orbits.
\begin{co} \label{prac}
If $\alpha$ is any quasi-action of any group $G$  
on an arbitrary metric space $X$ such that quasi-orbits ${\mathcal Q}(x_0)$
are quasi-geodesic spaces and also
quasi-isometrically embed in some proper metric space then
there exists a countable subset $S$ and $s>0$ such that the Rips graph
$\Gamma_s(S)$ is
connected, locally finite and quasi-isometric to ${\mathcal Q}(x_0)$,
along with a cobounded quasi-action $\beta$ on $\Gamma_s(S)$
such that the map from $\Gamma_s(S)$ to $X$ given
by sending any point to a nearest vertex and then to that point in
$S\subseteq X$
is a reduction of $\alpha$ to $\beta$.

If $\alpha$ is a cobounded quasi-action then this subset $S$ is coarse dense
in $X$ and $\beta$ is equivalent to $\alpha$.
\end{co}
\begin{proof}
  Without loss of generality by Theorem \ref{gnact} (i), we assume that
$\alpha$ is an isometric action on $X$ whose orbits have the same two
properties as our quasi-orbits ${\mathcal Q}(x_0)$. We then perform the
standard reduction of $\alpha$ to an isometric action on $O:=Orb(x_0)$ by
restriction, thus we have a cobounded (transitive) action. Now for badly
behaved orbits $O$ the Rips graph $\Gamma_s(O)$ need not be locally finite
(indeed $O$ need not even be a countable set). However $O$ does satisfy the
  conditions of Theorem \ref{prpr} by Lemma \ref{qiepr}, so we can
  apply this theorem to $O$. We obtain a coarse dense countable
  subset $S$ of $O$, along with $s>0$ which ensures that the map from
the locally finite (and connected) graph
  $\Gamma_s(S)$ to $O$ mentioned in the statement is a quasi-isometry,
  which we will call $r$, and
  thus a quasi-isometric embedding from $\Gamma_s(S)$ to $X$. Note that
  although we have an isometric action of $G$ on $O$ given by restriction
  of $\alpha$, we are not saying that $S$ is $G$-invariant and so we have
  no natural action on $\Gamma_s(S)$. Nevertheless 
  we take a quasi-inverse of $r$ which we call $q$, with $q$ mapping
  from $O$ to $\Gamma_s(S)$, so we can
now apply Example \ref{ex2} (3) to get our quasi-action $\beta$
on $\Gamma_s(S)$ which is equivalent to the restriction action of $\alpha$
on $O$, thus $\beta$ is cobounded too
and is also a reduction of our original quasi-action.

If this original quasi-action were cobounded then this is preserved throughout,
so $O$ will be coarse dense in $X$ and thus $S$ is too. Also our
quasi-isometric embedding of $\Gamma_s(S)$ to $X$ is a quasi-isometry
in this case, thus
this quasi-action $\beta$ is equivalent to the original quasi-action.
\end{proof}

We now have our quasi-action $\beta$ on the locally finite graph $\Gamma_s(S)$
and we can assume it sends vertices to vertices, as mentioned in Example
\ref{ex2} (3). However elements of $G$ will not give rise to automorphisms
of $\Gamma_s(S)$ under $\beta$ in general, because a vertex can be mapped
to another vertex with a completely different valence. This is somewhat
unfortunate because if the subset $S$ so obtained in the above proof
were equal to the orbit $O$ then (as
mentioned in subsection 3.3) we would get a genuine action of $G$ on this
Rips graph by
automorphisms, hence isometries, and this action would be vertex transitive,
implying that our Rips graph is not just locally finite but even of bounded
valence. However having a cobounded quasi-action on a locally
finite graph certainly does not imply that the graph has bounded valence.
For instance, take the regular tree $T_4$ and add an arbitrary
number of extra
edges at each vertex to create a tree $T$ in which $T_4$ is coarse dense,
hence they are quasi-isometric. Now turn the standard action of the
free group $F_2$ on $T_4$, which is cobounded, into an equivalent quasi-action
on $T$ using Example \ref{qunew}.

However we would be happy replacing our quasi-action on a locally finite
graph with an equivalent quasi-action on a bounded valence graph (so going
back from $T$ to $T_4$ in the above example). Whilst we know that
we cannot do this in general, we can if the quasi-action is cobounded.
\begin{prop} \label{blls}
  If $X$ is a metric space which is coarsely proper and admits
  some quasi-action of a group $G$ which is cobounded then $X$ is
  uniformly coarsely proper.
\end{prop}
\begin{proof}
  As being (uniformly) coarsely proper is a quasi-isometric invariant,
we can use Theorem \ref{gnact} (i) to assume that $G$ is acting by isometries.
  This action is also cobounded, so on picking a basepoint $x_0\in X$
  there exists $D\geq 0$ where for any $x\in X$ we can find $g\in G$
  such that $d(g(x_0),x)\leq D$. Now $X$ being coarsely proper implies that
  there is $R_0\geq 0$ such that for all $R\geq 0$, the ball $B(x_0,R)$
  can be covered by finitely many, say $M_R$, balls of radius $R_0$. But
  this means that the ball $B(g(x_0),R)$ can be covered by the same number
  of balls of radius $R_0$ since $G$ acts by isometries. We have
  $B(x,R)\subseteq B(g(x_0),R+D)$ so that any ball in $X$ of radius $R$
  can be covered by $M_{R+D}$ balls of radius $R_0$ and so $X$ is uniformly
  coarsely proper. (This can also be established by working directly
  with the $(K,\epsilon,C)$ quasi-action rather than quoting Theorem
  \ref{gnact} (i) and keeping track of the constants that emerge in terms of
  $K,\epsilon,C$.)
\end{proof}  
\begin{co} \label{pracb}
If $\alpha$ is any quasi-action of any group $G$  
on an arbitrary metric space $X$ such that quasi-orbits ${\mathcal Q}(x_0)$
are quasi-geodesic spaces and also
quasi-isometrically embed in some proper metric space then
there exists a connected graph $\Gamma'$ of bounded valence which
is quasi-isometric to ${\mathcal Q}(x_0)$,
along with a cobounded quasi-action of $G$ on $\Gamma'$ which
is a reduction of $\alpha$.

If $\alpha$ is a cobounded quasi-action then this reduction is an
equivalence.
\end{co}
\begin{proof}
On applying Corollary \ref{prac} we emerge with a reduction of our original
quasi-action $\alpha$ to a cobounded quasi-action $\beta$ on a connected
locally finite graph $\Gamma_s(S)$. Now such a graph is certainly a coarse
proper metric space, so the cobounded quasi-action tells us that $\Gamma_s(S)$
is a uniformly coarse proper space by Proposition \ref{blls}. Moreover
$\Gamma_s(S)$ is certainly a (quasi-)geodesic space so it is quasi-isometric
to a connected graph $\Gamma'$ of bounded valence. We can now use
Example \ref{qunew} again to transfer $\beta$ to some cobounded quasi-action
of $G$ on $\Gamma'$ which is still a reduction of $\alpha$ (and an
equivalence if $\alpha$ were itself cobounded).
\end{proof}
Note that the sufficient conditions given the the above hypothesis are
also necessary.

\section{Group (quasi-)actions on hyperbolic spaces}

We now specialise to the case where our space $X$ is hyperbolic, that 
is a geodesic metric space
satisfying any of the equivalent definitions of $\delta$ - hyperbolicity.
However no further
conditions such as properness of $X$ will be assumed. We have the
Gromov boundary $\partial X$ of $X$ where a boundary point is
an equivalence class of sequences in $X$ under an equivalence relation
which uses the Gromov product on $X$.
This turns $X\cup \partial X$ into a topological space with $X$ (whose
subspace topology is the original topology on $X$) a dense subspace
(though $X\cup\partial X$ need not be compact if $X$ is not proper).

We now consider a group $G$ acting on $X$ by isometries.
As $X$ is hyperbolic this action extends to an
action of $G$ by homeomorphisms (though not isometries)
on $\partial X$. We then obtain the limit set
$\partial_G X\subseteq \partial X$. This is defined to be the intersection
of the set of accumulation points in $\partial X$
of some (any) orbit $G\cdot x_0$ for $x_0\in X$. Consequently $\partial_GX$
is $G$-invariant and we have $\partial_H X\subseteq \partial_G X$ if $H$ is
a subgroup of $G$. Here we note that if the action is cobounded then any
sequence in $X$ representing a point at infinity can be replaced by an
equivalent sequence lying in $G\cdot x_0$, so that $\partial_GX=\partial X$
in this case.

Let us first consider cyclic groups, by taking an arbitrary isometry $g$
of some hyperbolic space $X$ and examining $\partial_{\langle g\rangle} X$.
We can summarise the facts we will need in the following well known
proposition.
\begin{prop} \label{cyc}
(i) If $g$ is elliptic then $\partial_{\langle g\rangle} X=\emptyset$. This
is always the case if $g$ has finite order but might also occur for elements
with infinite order. In either case $g$ might fix many or no points on
$\partial X$.\\
(ii) If $g$ is loxodromic then $\partial_{\langle g\rangle} X$ consists
of exactly 2 points $\{g^\pm\}$ for any $x\in X$ 
and this is the fixed point set of $g$ (and $g^n$
for $n\neq 0$) on $\partial X$. Moreover for any $x\in X$ we have
$g^n(x)\rightarrow g^+$ and $g^{-n}(x)\rightarrow g^-$ in $X\cup\partial X$
as $n\rightarrow\infty$.\\
(iii) If $g$ is parabolic then $\partial_{\langle g\rangle} X$ consists
of exactly 1 point and again this is the fixed point set of $g$ (and $g^n$
for $n\neq 0$) on $\partial X$.\\
\end{prop}

Moving back now from cyclic to arbitrary groups, for any group $G$ 
(not necessarily finitely generated)
acting by isometries on an arbitrary hyperbolic space $X$, we have the
Gromov classification dividing possible actions into five very different
classes. In the case that $G=\langle g\rangle$ is cyclic, the first three 
classes correspond to the three cases in Proposition \ref{cyc} and
the next two do not occur (for these facts and related references, see 
for instance \cite{abos}):\\
\hfill\\
(1) The action has {\bf bounded orbits}. This happens exactly when
$\partial_G X$ is empty, in which case all elements are elliptic.
(Note that if $\partial_GX=\emptyset$ then it requires further work to
argue that orbits are bounded, because we cannot assume in general
that $X\cup\partial X$ is (sequentially) compact.)
However if all elements are elliptic then we can also be in case (2).\\
\hfill\\
(2) The action is {\bf parabolic} (or horocyclic),
meaning that $\partial_G X$ has exactly
one point $p$. Note that, despite the name, we
can still be in this case but with $G$ consisting only of elliptic elements.
However any non elliptic element must be parabolic with limit
set $\{p\}$.\\
\hfill\\
(3) The action is {\bf lineal}, meaning that
$\partial_G X=\{p,q\}$ has exactly 2 points.
These points might be swapped, so for this paper we use the shorthand that
the action is of type (3)$^+$ if they are 
pointwise fixed by $G$ and of type (3)$^-$ if not
(thus a group $G$ with a type (3)$^-$ action has a unique subgroup
of index 2 with a type $(3)^+$ action). In both cases there
will exist some loxodromic element in $G$ with limit set $\{p,q\}$ and
indeed all loxodromics have this limit set. Moreover $G$ contains
no parabolics (such an element $g$ would have to fix $p$ or $q$ but if
it were $p$ then $g$ would move $q$ outside $\{p,q\}$ which is
$G$-invariant) but there might be many elliptic elements in $G$.\\
\hfill\\
(4) The action is {\bf quasi-parabolic} (or focal).
This says that the limit set
has at least 3 points, so is infinite, but there is some point
$p\in\partial_G X$ which is globally fixed by $G$. This implies that
$G$ contains a pair of loxodromic elements with limit sets
$\{p,q\}$ and $\{p,r\}$ for $p,q,r$ distinct points.\\
\hfill\\
(5) The action is {\bf general}: the limit set is infinite and we have two
loxodromic elements with disjoint limit sets, thus by ping pong high powers
of these elements will generate a quasi-embedded copy of the free group
$F_2$.\\

We refer to these descriptions, or just the numbers (1) to (5)
as the {\bf hyperbolic type} of the action.
Suppose that $G$ has an isometric action on the hyperbolic space
$X$ which reduces to another isometric action on the geodesic space $Y$.
Thus we have a coarse $G$-equivariant quasi-isometric
embedding $F$ from $Y$ to $X$. As $Y$ is a geodesic space which
quasi-isometrically embeds in a hyperbolic space, it too is hyperbolic
(\cite{bh} Theorem III.H.1.9). Now we have already said in Proposition
\ref{qpres} that
the isometry type of individual elements will be the same in both actions
and indeed the hyperbolic type of the action will be the same too.
This is shown in \cite{abos} Lemma 4.4
when $F$ is also a quasi-isometry but the proof goes through for this case
too. Moreover it is also mentioned that the parity of type (3) actions, that
is whether they are (3)$^+$ or (3)$^-$, is preserved too by this process.

We quote the following proposition which tells us how the orbits look
in any of these five types of action. This is Proposition 3.2 in
\cite{cdcmt}. (Note they use the term quasi-convex, which is equivalent to
our term quasi-geodesic space by Theorem \ref{qugd}. This is
not equivalent to the usual meaning of quasi-convex in arbitrary
geodesic metric spaces, but it is for hyperbolic spaces using the Morse
lemma.)
\begin{prop} \label{orblk}
  For a type (2)
  (parabolic) action, no orbit is a quasi-geodesic space. In any type (1)
  (bounded), type (3) (lineal) or type (4)
  (quasi-parabolic) action, all orbits are quasi-geodesic spaces.
  In a type (5) (general) action, orbits may or
  may not be quasi-geodesic spaces.
\end{prop}  

We now give some examples that illustrate this classification and which
will be relevant later.\\
{\bf Example 1}:
Any infinite finitely generated group acts by isometries on its
Groves - Manning combinatorial
horoball (see \cite{mnarth} 3.1) which is a hyperbolic space. This is an
action of type (2)
with the finite order elements acting elliptically (as always)
and all infinite order elements acting parabolically. The space is
a locally finite (but not bounded valence)
hyperbolic graph with $G$ acting metrically properly and freely.
Moreover any countable group also has such an action by embedding
it in a finitely generated group.

For these examples we then have that no orbit is a quasi-geodesic space by
Proposition \ref{orblk}. But by Proposition \ref{fgn} (ii) orbits are coarse
connected exactly when the group is finitely generated. Thus we obtain lots
of examples of bad orbits and lots of examples of orbits which are neither
good nor bad, according to Definition \ref{gdbd}.\\
\hfill\\
{\bf Example 2}:
It might be expected that no action of type (2) exists on a simplicial tree
because there are no parabolic isometries. This is true if the group is
finitely generated by \cite{ser} I.6.5 (indeed for a generating set
$\{g_1,\ldots ,g_n\}$ it is enough to check that each $g_i$ and $g_ig_j$ is
elliptic, whereupon the action is of type (1) with a global fixed point).
However, given any countably infinite group $G$
which is locally finite (every finitely generated subgroup is finite),
we can obtain an action of $G$ on a locally finite tree with unbounded orbits
where every element acts elliptically. This is sometimes called the
coset construction and is described in \cite{ser} I.6.1 or in  \cite{bh}
II.7.11.
The basic idea is to express $G$ as an increasing union
$\cup_{i=0}^\infty G_i$ of finite subgroups $G_i$ with $G_0=\{e\}$. The 0th
level vertices are the elements of $G$, the 1st level vertices are the
cosets of $G_1$ and so on. The stabiliser of a vertex in level $i$
has size $|G_i|$ and its valence is
$[G_i:G_{i-1}]+1$ (and 1 in level zero). Thus this action is metrically
proper and we even have a bounded valence tree
if $[G_i:G_{i-1}]$ is bounded.\\
\hfill\\
{\bf Example 3}: If $G$ acts isometrically on a hyperbolic space $X$
then we can ask: does $G$ have an isometric action on some hyperbolic space
where the action is of the same type but is cobounded? By Proposition
\ref{orblk} no action of type (2) on any hyperbolic space can be cobounded
because then an orbit would be quasi-isometric to this space and so
would be quasi-geodesic. Moreover the group $G$ in Example 2 is torsion,
so cannot admit actions of types (3), (4) or (5) where loxodromics are
present, which means that its only cobounded actions on hyperbolic spaces
are bounded actions on bounded spaces. However it is shown in
the appendix of \cite{bargen} that this question has a positive answer
if the action is not of type (2). But if the orbits are
quasi-geodesic spaces then we can merely form the Rips graph
$\Gamma_s(G\cdot x_0)$ using the orbit of any $x_0\in X$. We have seen that
the resulting action of $G$ on this graph, which is clearly cobounded,
is a reduction of the original action for $s$ large enough and if so
then $X$ will be hyperbolic.
Thus the two actions will have many properties in common including
being of the same type. By Proposition \ref{orblk}, if the original
action is of types (1), (3), (4) or of type
(5) with quasi-geodesic orbits
then we have a more direct construction of a cobounded action, but this
will not deal with type (5) actions if the orbits are not quasi-geodesic
spaces. Indeed in this case by Lemma \ref{ccnq} (ii), no reduction to
a cobounded isometric action on a hyperbolic space exists so here the
construction has to be more involved and will necessarily change the coarse
behaviour of the orbits.

\section{Coarse geometry of trees}

We now turn to the spaces that really interest us, namely simplicial
trees and those metric spaces which are quasi-isometric to simplicial
trees (we reserve the term quasi-tree for such spaces which are also
geodesic).

We first look at convex closures. Given any geodesic metric space $X$,
we say that a subset $C$ of $X$ is {\bf convex} if any geodesic
between two points in $C$ also lies entirely in $C$. In a tree $T$, convex
and connected subsets are the same thing because geodesics are unique
and removing any point from $T$ makes it disconnected
(here we can use connectedness and path connectedness interchangeably).
For any subset
$S$ of $X$, we can define the {\bf convex closure} $CCl(S)$ of $S$ in $X$ to
be the smallest convex subset of $X$ containing $S$. As an arbitrary
intersection of convex subsets is convex and $X$ is itself convex, this
set $CCl(S)$ exists and is well defined. In general $CCl(S)$
can be bigger than merely taking the points of $S$ union all geodesics in $X$
between points of $S$. But in a simplicial tree $T$ equipped with the path
metric $d_T$, these sets are equal.
This can be seen by taking four arbitrary
points $a,b,c,d\in S$ and then considering a point $x$ on the geodesic $[a,b]$
and a point $y$ on $[c,d]$. By drawing the various cases of how $[a,b]$
and $[c,d]$ meet (if at all) and the possible locations of $x$ and $y$,
in each case we can find a geodesic between two of the points $a,b,c,d$ which
contains both $x$ and $y$, hence also $[x,y]$.

Moreover if $S$ is a subset
of the vertices of $T$ (which is the case we will be using) then $CCl(S)$
is naturally a simplicial graph which itself has no closed loops, thus
it is a subtree of $T$ and its path metric is the same as the
restriction of $d_T$. We might now ask: when is an arbitrary
subset $S$ of $T$ (given the subspace metric) itself quasi-isometric to some
simplicial tree and when can we take that tree to be $CCl(S)$?
Here we can give a complete answer.
\begin{thm} \label{sstt}
Let $T$ be a simplicial tree with path metric $d_T$ and let $S$ be any
non empty subset of $T$, equipped with the subspace metric. Then the
following are equivalent.\\
(i) $S$ is coarse connected.\\
(ii) $S$ is a quasi-geodesic space.\\
(iii) $S$ is quasi-isometric to some simplicial tree.\\
(iv) The natural inclusion of $S$ in its convex closure $CCl(S)$ is a
quasi-isometry.
\end{thm}  
\begin{proof}
First suppose that $S$ is a subset of vertices of $T$, in which case it was
mentioned that the convex closure $CCl(S)$ is a subtree of $T$. Therefore
we need only show that (i) implies (iv). But $S$ and $CCl(S)$ both inherit
their respective metrics from $T$, so we just need that $S$ is coarse
dense in $CCl(S)$.

Therefore take some large $K>0$ and assume that $S$ is not $K$-coarse
dense in $CCl(S)$. This would mean that
we have $x,y\in S$ and a geodesic $[x,y]$ in $T$ between these points,
but there is $z\in [x,y]$ having distance more than $K$ to any point
of $S$. Now removing the point $z$ from $T$ makes the resulting
space disconnected and with $x,y$ in different connected components.
Let us intersect each of the connected components of $T\setminus \{x\}$ with
$S$ to form the sets $A_i$ (for $i$ in some indexing set $I$) which partition
$S$ and suppose that $x\in A_0$ say. Note that each $A_i$ is contained in
$T\setminus B(z,K)$ because $S$ is.
On taking two arbitrary points $a\in A_0$ and $b\in A_j$ for $A_0\neq A_j$ and
the unique geodesic in $T$ between them, we see that $[a,b]$ has to pass
through $z$ as $a$ and $b$ are in different components of $T\setminus\{z\}$.
We have $d(a,z)>K$ and the same for $d(b,z)$, hence
\[d(a,b)=d(a,z)+d(z,b)>2K.\]
Thus $S$ is not $2K$-coarse connected, because on starting at $x\in A_0$
and taking jumps in $S$ of length at most $2K$, we remain in $A_0$ so never
reach $y\in S\setminus A_0$. We can now let $K$ tend to infinity for a
contradiction.

In the case where $S$ is an arbitrary subset of points rather than just
vertices, it is always true that (ii) implies (i). Moreover
the argument just given works identically here so that (i) implies (iv)
here too and $CCl(S)$ is certainly a geodesic space so (iv) implies (ii).
The only difference now is that $CCl(S)$ might not strictly be a simplicial
tree because edges might get ``chopped off'' and so not have length 1,
but these can be removed to obtain a simplicial tree which is quasi-isometric
to $CCl(S)$, thus (iv) implies (iii) which certainly implies (ii).
\end{proof}   

This gives us the following useful corollaries which apply to all spaces
that are quasi-isometric to some simplicial tree, without any finiteness
assumptions.

\begin{co} \label{nopar}
Suppose that $\alpha$ is a quasi-action of an arbitrary group
$G$ on a metric space $X$ which is quasi-isometric to some simplicial
tree. If $\alpha$ has coarse connected quasi-orbits then any quasi-orbit
is also quasi-isometric to a tree and hence is
a quasi-geodesic space.
\end{co}
\begin{proof} Change $\alpha$ into an equivalent quasi-action $\alpha'$
  on the
  tree $T$, which will still have coarse connected quasi-orbits by Lemma
  \ref{ccnq} (i). By Theorem \ref{sstt} any quasi-orbit of $\alpha'$
  on $T$ will be quasi-isometric to some tree. Thus the same holds for
the quasi-orbits of $\alpha$ too, by Lemma \ref{ccnq} (i) again.
\end{proof}
Note that the examples both at the end of Section 3 and 
Section 5 Example 2 each demonstrate an isometric action of
a (necessarily infinitely generated) group on a bounded valence simplicial
tree where the orbits are not coarse connected and so this
conclusion fails.

Our next two corollaries generalise the
result in \cite{mnlms} Section 3 that any cyclic group $\langle g\rangle$
acting by isometries, or indeed quasi-acting,
on a quasi-tree is either elliptic or loxodromic.
\begin{co} \label{nfgp}
  No isometric action of
  any finitely generated group $G$ on any quasi-tree can be a type (2)
  (parabolic) action.
\end{co}
\begin{proof}
  By Proposition \ref{fgn} (i) the orbits are coarse connected, so are
quasi-geodesic spaces by Theorem \ref{sstt}. If this were a type (2) 
action, Proposition \ref{orblk} would be contradicted.
\end{proof}
\begin{co} \label{serc}
Suppose that $G$ is a finitely generated group with a quasi-action on some
metric space that is quasi-isometric to a simplicial tree. If no element
quasi-acts loxodromically then this quasi-action has bounded quasi-orbits.
\end{co}
\begin{proof}
  Change the quasi-action into an equivalent isometric action on some geodesic
  space by Theorem \ref{gnact} (ii). This new space
  is still quasi-isometric to a simplicial tree thus is a hyperbolic space.
Also  the isometry type of
every element will be preserved, so this action must be of hyperbolic
type (1) or (2). But type (2) is ruled out by Corollary \ref{nfgp}, thus
the original quasi-action was bounded too.
\end{proof}
{\bf Example}: We can see this corollary as a parallel result for quasi-actions
on quasi-trees of Serre's result mentioned earlier that a 
finitely generated group $\langle g_1,\ldots ,g_n\rangle$ acting on
a simplicial tree has a global fixed point if
each $g_i$ and $g_ig_j$ are elliptic. But there is no direct equivalent
of this result for quasi-trees, even for isometric actions.
This can be seen by quoting
\cite{balas} which implies that any acylindrically hyperbolic group has an
unbounded action on a quasi-tree (here the quasi-tree is a graph but is
not locally finite). Therefore we start with a hyperbolic
group generated by $x_1,\ldots ,x_n$ (say the free group $F_n$) and create
a non elementary hyperbolic quotient where all group words up to a given
length in $x_1,\ldots ,x_n$ have finite order, by quotienting out a very
high power of these elements. Thus the quotient will also have an unbounded
action on a quasi-tree but all of these elements will act elliptically.

We also note here some parallels and differences
between groups acting on quasi-trees
and groups acting on bounded valence (as opposed to locally finite)
hyperbolic graphs. There are also no parabolic isometries in the latter case
and recently it was shown in \cite{dmht} using Helly
graphs that a finitely generated group $G$ acting on any bounded valence
hyperbolic graph with every element elliptic
must act with bounded orbits. Thus Corollary \ref{nfgp} holds for these
spaces because the absence of parabolic elements means that every
element would be elliptic in a type (2) action, as does Corollary \ref{serc}
for these spaces in the case of isometric actions. However the implication
of (i) implies (ii) in Theorem \ref{sstt} as well as Corollary \ref{nopar}
are both false here, as can be seen
by taking a finitely generated subgroup of a hyperbolic group that is
not quasi-convex, say the fibre subgroup of a closed fibred hyperbolic
3-manifold group. The subgroup acts isometrically on a bounded valence
hyperbolic graph (the Cayley graph of the whole group) but with orbits
that are not quasi-convex and thus (by the Morse lemma) not quasi-geodesic
spaces.\\
\hfill\\
{\bf Example}: We also have that Corollary \ref{serc} fails for spaces that
are bounded valence hyperbolic graphs if we allow quasi-actions rather than
just actions.
For instance, \cite{watr} gives an isometric action of a 2-generator
group $G$ on the real hyperbolic space $\mathbb H^4$ which is a type
(2) action, hence unbounded, but where every element is elliptic. On
taking $G_0$ to be a cocompact lattice in $Isom(\mathbb H^4)$ and $\Gamma_0$
to be the Cayley graph of this lattice with respect to some finite
generating set,
we have that $\Gamma_0$ is a bounded valence hyperbolic graph which is
quasi-isometric to $\mathbb H^4$. Therefore our isometric action of $G$
on $\mathbb H^4$ is equivalent to an unbounded quasi-action of $G$ on $\Gamma_0$
where every element quasi-acts elliptically, but there is no equivalent
isometric action on $\Gamma_0$ itself. Indeed there is no reduction
to an isometric action on any bounded valence graph as this would
be hyperbolic, thus contradicting \cite{dmht}.

Corollary G of \cite{mrgo} states that if $X$ is a proper non elementary
hyperbolic space and we have a cobounded quasi-action $\alpha$
of $G$ on $X$
which does not fix a point of $\partial X$ then $\alpha$ is equivalent
to an isometric action on either a rank one symmetric space of non-compact
type or on a (connected) locally finite graph. However it can happen
that even if $X$ itself is a non elementary hyperbolic
graph of bounded valence, we can have a cobounded quasi-action $\alpha$
on $X$ which does not fix any point of $\partial X$ but where there is
no isometric action on $X$ or on any locally finite graph
which is equivalent to $\alpha$. Indeed this is the situation whenever
the quasi-action $\alpha$ contains a parabolic element, such as $PSL(2,\R)$
quasi-acting via Example \ref{qunew}
on the Cayley graph of a closed hyperbolic surface group which is 
quasi-isometric to $\mathbb H^2$. (For a finitely generated example,
we could take a cocompact lattice in $PSL(2,\R)$ thus giving us a
cobounded isometric action on $\mathbb H^2$ and then throw in a parabolic
element.)

We now consider quasi-isometries
between quasi-trees which are locally finite graphs equipped with
the path metric. In general there is no relationship between the valences
of quasi-isometric trees: for instance all regular trees of valence at least
3 are quasi-isometric and we can certainly have bounded valence trees which
are quasi-isometric to trees which are not of bounded valence (or not even
locally finite). If however we have some constraints on our trees
then the restrictions are much stronger. The following is the crucial lemma.
\begin{lem} \label{lfbd}
  Suppose that we have a $(K,\epsilon,C)$ quasi-isometry from a locally finite
  graph $\Gamma$ equipped with the path metric to some simplicial tree
  $T$ and which sends vertices to vertices.
  Then there exists a subtree $S$ of $T$ which is quasi-isometric
  to $T$ such that for any vertex $s\in S$ with valence $d_s$,
  there is a vertex $p_s$ of $\Gamma$ satisfying
  \[F(p_s)\in S\mbox{ and }  
  d_S(F(p_s),s)\leq K+\epsilon,\]
such that the closed ball in $\Gamma$ around $p_s$ of radius
$2K^2+3K\epsilon$ contains at least $d_s$ vertices.

Furthermore (assuming that $T$ is unbounded)
for any vertex $t\in T$, 
  there is a vertex $p_t$ of $\Gamma$ with  $d_T(F(p_t),t)\leq K+\epsilon$,
  such that the number of vertices in the closed ball in $\Gamma$
  around $p_t$ of
  radius $2K^2+3K\epsilon$ is at least the number of unbounded components
  of $T\setminus \{t\}$.
\end{lem}
\begin{proof}
  If $V$ is the vertex set of $\Gamma$ then the subtree $S$ of $T$ is
  defined simply as the convex closure of $F(V)$ in $T$. As $F(V)$ is
  coarse dense in $T$ so is $S$, hence it is quasi-isometric to $T$
  by inclusion. 

  Let us take any vertex $s\in S$.
  Removing $s$ from $T$ splits $T\setminus \{s\}$ into a number (possibly
infinite) of connected components. Consider such a component $C_0$
which intersects $S$. This means that there must be some vertex $v\in V$
with $F(v)$ lying in $C_0$, as otherwise $F(V)$ would lie in the
convex set $T\setminus C_0$ so $S$ would too. But similarly there must be
some vertex $w\in V$ with $F(w)$ lying outside $C_0$ (possibly $F(w)=s$)
as otherwise $C_0$ minus the open edge from $s$ to $C_0$ would also be convex
with $F(V)$ lying in it.

Thus take a geodesic in $\Gamma$ from $v$ to $w$ and let $v=v_0,v_1,\ldots ,
v_n=w$ be the vertices passed through in order. There exists $i$ such that
$F(v_i)$ is not in $C_0$ but $F(v_{i-1})$ is, so as $d_\Gamma(v_{i-1},v_i)=1$
we obtain
\[d_S(F(v_{i-1}),s)\leq d_S(F(v_{i-1}),F(v_i))\leq K+\epsilon\]
as $s$ is the nearest vertex in $T$ that is outside $C_0$.
Now each component $D$ of $S\setminus \{s\}$ corresponds to exactly
one component of $T\setminus\{s\}$ that meets $S$, so for each such $D$ we
can pick one vertex $u_D$ of $\Gamma$ with $F(u_D)$ in $D$
and with $d(F(u_D),s)\leq K+\epsilon$,
so that for components $D,D'$ we have
\[d_S(F(u_D),F(u_{D'}))\leq 2K+2\epsilon\mbox{ and hence }d_\Gamma(u_D,u_{D'})
  \leq 2K^2+3K\epsilon.\]
Now the set $\{u_D\}$ has cardinality equal to the number of components
of $S\setminus \{s\}$, which is just the valence $d_S$, and on taking
$p_S$ to be any element of this set $\{u_D\}$ we have that $d_S(F(p_S),s)\leq
K+\epsilon$ and $d_\Gamma(p_s,u_D)\leq 2K^2+3K\epsilon$, hence the lower
bound for the number of vertices in this closed ball.

We then run through the same argument in $T$ as opposed to $S$, where $t$ is
any vertex of $T$, and this time we consider the connected components of
$T\setminus\{t\}$. Although it could be that there are components $C_0$
which are disjoint from $F(V)$, this cannot occur if $C_0$ is unbounded
because $F$ is $C$-coarse onto. Thus on picking a vertex $v'$ in this
unbounded component $C_0$ that has distance more than $C$ from the vertex
$t$, we have that the closed ball of radius $C$ around $v'$ lies
completely in $C_0$. We now run through the rest of the argument, where
$D$ is now an unbounded component of $T\setminus \{t\}$.
\end{proof}

\begin{co} \label{lftr}
  A locally finite (respectively bounded valence) connected graph
$\Gamma$  equipped with the path metric which is quasi-isometric to some
simplicial tree $T$ is also quasi-isometric to a locally finite
(respectively
  bounded valence) simplicial tree.
\end{co}
\begin{proof}
We may assume that any quasi-isometry sends vertices to vertices.  
On applying Lemma \ref{lfbd} to $\Gamma$ and $T$, we obtain a simplicial
tree $S$ quasi-isometric to $\Gamma$ where the valence of a vertex $s\in S$
is bounded above by the number of vertices of $\Gamma$ lying in a closed
ball about some vertex $v$ of $\Gamma$ and of constant radius. This
number is finite if $\Gamma$ is
locally finite and is bounded above independently of $v$ if $\Gamma$ has
bounded valence.
\end{proof}

\begin{ex} \label{root}
  Take any tree $T$ where every vertex has valence at least
two (possibly infinite).
  Then $T$ is quasi-isometric to some locally finite (respectively bounded
  valence) graph $\Gamma$ if and only if it is locally finite
  (respectively has bounded valence) itself. This
is because on applying Lemma \ref{lfbd},
for any vertex $t$ of $T$ the number of unbounded components of
$T\setminus\{t\}$ is a lower bound for the number of vertices of $\Gamma$
lying in a ball of some fixed radius, but with centre a varying vertex of
$\Gamma$. But as $T$ has no leaves, the number of unbounded components
of $T\setminus \{t\}$ is just the valence of $t$ which must therefore
be finite. If further $\Gamma$ has bounded valence then this number
is bounded above as the centre varies, thus $T$ also has bounded valence.
\end{ex}

\section{Reducing our quasi-actions}

We are now in a position to apply the material in the last few sections
to the completely general case of an arbitrary quasi-action of any group on an
arbitrary space where the quasi-orbits look vaguely tree-like: that is there
is a quasi-isometric embedding of a quasi-orbit in some simplicial tree.
Our aim is to reduce this quasi-action to a cobounded quasi-action on a
bounded valence tree. Clearly a necessary condition from earlier is that
quasi-orbits must be coarse connected. Another necessary condition
is that as a bounded valence tree is proper, 
we will require that quasi-orbits quasi-isometrically embed in
some proper metric space (which need not be this simplicial tree).
This will be all we require.

\begin{thm} \label{gnthm}
  Let $\alpha$ be any quasi-action of any group $G$ on an arbitrary metric
  space $X$. Then $\alpha$ can be
  reduced to some cobounded quasi-action on a bounded valence simplicial
  tree if and only if the quasi-orbits ${\mathcal Q}(x_0)$ of $\alpha$
  quasi-isometrically embed in some
  proper metric space, quasi-isometrically embeds in some
  simplicial tree and are coarse connected.

  If $\alpha$ is itself a cobounded quasi-action then we can replace
  ``can be reduced to'' in the above with ``is equivalent to''.
\end{thm}
\begin{proof}
  First note that by Proposition \ref{orgi}, whenever we 
  reduce some quasi-action $\alpha$ to another quasi-action, any quasi-orbit
  continues to be coarse connected and to quasi-isometrically embed in both
  a proper metric space and a tree. In particular the conditions given on
  ${\mathcal Q}(x_0)$ are necessary for such a quasi-action to exist.

  Now suppose that these conditions do hold.
  The image in this tree of our quasi-orbit ${\mathcal Q}(x_0)$ will also
  be coarse connected, thus by Theorem \ref{sstt} ${\mathcal Q}(x_0)$ will
itself be quasi-isometric to a tree and hence is
  a quasi-geodesic space. Thus by Corollary \ref{pracb} we obtain 
  a reduction of $\alpha$ to a cobounded quasi-action $\beta$ on some
  connected graph $\Gamma$ of bounded valence, which is an equivalence if
  $\alpha$ is cobounded. Moreover $\Gamma$ is quasi-isometric to a
  quasi-orbit, so also to some simplicial
  tree $T$. Then by Corollary \ref{lftr} we can take this tree $T$ to be have
  bounded valence and we can again transfer our quasi-action $\alpha$ to an
  equivalent one on $T$ which is still cobounded.
  
  If $\alpha$ is cobounded then any reductions in this proof will be
  equivalences, thus we will end up with an equivalent quasi-action.   
\end{proof}  

\section{Unbounded orbits on quasi-trees}

So far we have taken quasi-actions whose quasi-orbits are 
quasi-isometric to trees and which quasi-isometrically embed in proper
spaces. We have reduced these to cobounded quasi-actions
on bounded valence trees,
but we have not turned any of these into genuine isometric actions.
This will be done for us by applying the following strong result.
\begin{thm} \label{mswt} (\cite{msw} Theorem 1)\\
If $G\times T\rightarrow T$ is a cobounded quasi-action of a group $G$
on a bounded valence bushy tree $T$, then there is a bounded valence, bushy
tree $T'$, an isometric action $G\times T'\rightarrow T'$, and a
coarse $G$-equivariant quasi-isometry from the action of $G$ on $T'$ to
the quasi-action of $G$ on $T$.
\end{thm}
The only term in this statement which needs defining is that of a
{\bf bushy} tree, which is a simplicial tree $T$ 
where there exists $b\geq 0$ (the bushiness
constant) such that for any point $x\in T$,
$T\setminus B(x,b)$ has at least three unbounded
components. This property can be seen to be preserved by quasi-isometries
between bounded valence trees, thus the second mention of bushy when
describing $T'$ is
strictly speaking redundant. We first look to see when we have a bushy tree. 
\begin{prop} \label{mswp}
Suppose that a group $G$ has a cobounded quasi-action on a
bounded valence tree $T$. If $T$ has at least 
three ends then $T$ is a bushy tree. 
\end{prop}
\begin{proof}
This is (the proof of) \cite{mngeom} Theorem 4.20.
\end{proof}
Moreover as stated in \cite{msw},
any two bushy bounded valence trees are quasi-isometric (a proof follows
from \cite{lndq} Lemma 2.6). As for quasi-trees,
a useful and well known criterion to tell whether a geodesic metric
space $X$ is a quasi-tree is Manning's bottleneck criterion (see \cite{mnlms}).
This states that there is some number $C\geq 0$ (the bottleneck constant)
such that for every geodesic segment $[x,y]$ in $X$ and any point $z$ on
$[x,y]$, any path between $x$ and $y$ must intersect the closed ball
$B(z,C)$. Combining this with the above gives us what might be called the
``bushy bottleneck'' criterion
for when a bounded valence graph (equipped with the path metric) is
quasi-isometric to the 3-regular tree.
\begin{prop} A connected graph $\Gamma$ of bounded valence is quasi-isometric
to the 3-regular tree $T_3$ if and only if there is a constant $C\geq 0$
such that for all $z\in\Gamma$, we have $\Gamma\setminus B(z,C)$ has
at least three unbounded components and also for every geodesic segment
$[x,y]$ in $\Gamma$ and any point $z$ on $[x,y]$, any path in $\Gamma$
between $x$ and $y$ must intersect the closed ball $B(z,C)$. 
\end{prop}
\begin{proof} This criterion can be seen to be a quasi-isometry invariant
  amongst bounded valence graphs
  and is possessed by $T_3$. Now suppose we have a bounded valence graph
  $\Gamma$ satisfying this condition. Then $\Gamma$ is quasi-isometric to
  a tree by the bottleneck criterion and thus quasi-isometric to a
  bounded valence tree $T$ by Corollary \ref{lftr}. But $T$ also satisfies
  our ``bushy criterion for graphs'' as it is a quasi-isometric invariant,
  so $T$ is a bounded valence bushy tree which is therefore quasi-isometric
  to $T_3$ as mentioned above, hence so is $\Gamma$.
\end{proof}

We can now use this result and our theorem from the previous section
in combination with Theorem \ref{mswt}.
\begin{thm} \label{main}
Let $\alpha$ be any quasi-action with coarse connected
quasi-orbits of any group $G$ on an arbitrary metric space $X$.
Suppose that the quasi-orbits of $\alpha$ both quasi-isometrically embed
into a tree and quasi-isometrically embed into a proper metric space.
  Then exactly one of the following three cases occurs:\\
$\bullet$
$\alpha$ 
  reduces to some cobounded isometric action on a bounded valence simplicial
  tree which is bushy. This occurs if and only if quasi-orbits of
  $\alpha$ are not bounded subsets or quasi-isometric to $\R$.\\
$\bullet$
or $\alpha$ reduces to some cobounded quasi-action on the real line.   
This occurs if and only if quasi-orbits of $\alpha$ are quasi-isometric to
$\R$.\\
$\bullet$
or $\alpha$ reduces to the
trivial isometric action on a point. This occurs if and only if
quasi-orbits of $\alpha$ are bounded subsets.
In each of these three cases, this reduction is
  an equivalence of quasi-actions if and only if $\alpha$ is cobounded.
\end{thm}
\begin{proof}
  We apply Theorem \ref{gnthm} to obtain a reduction of $\alpha$ to
  (or an equivalent quasi-action if and only if $\alpha$ is cobounded)
  a cobounded quasi-action $\beta$ on a bounded valence simplicial tree
  $T$. If $T$ is bounded then we can take it to be a point and this
  occurs if and only if $\alpha$ is a bounded quasi-action.

Otherwise there must exist an element which quasi-acts loxodromically under
$\beta$, because if not then Corollary \ref{serc} (which applies to
quasi-actions of groups that need not be finitely generated as long as
the quasi-orbits are coarse connected) says that the quasi-action
  is bounded but it is also cobounded, thus $T$ would be a bounded tree.
In particular the case when $T$ has one end does not occur
  because we must have a quasi-isometrically embedded
  copy $S$ say of $\Z$ in $T$ and hence also of $\R$. In the two ended
  case any self $(K,\epsilon,C)$ quasi-isometry of $T$ will send $S$
  to a set with the same endpoints which is at a bounded
  Hausdorff distance from $S$, by the Morse lemma.
  But our quasi-action $\beta$ is cobounded, therefore
  any point of $T$ is also at a bounded distance from $S$. Hence $T$
  is quasi-isometric to $\R$ and so we can assume $T=\R$ by transferring
 the quasi-action. 

Otherwise $T$ has at least three ends and $\beta$ is a cobounded quasi-action,
so Proposition \ref{mswp} applies, allowing us to use Theorem \ref{mswt}
to obtain a reduction of $\alpha$ to a genuine isometric action on some
bounded valence, bushy tree (which is an equivalence if $\alpha$ was
itself cobounded).
\end{proof}

Note that each of the three conditions in the statement of Theorem
\ref{main} can nullify the conclusion independently from the other two.
Without the
proper condition, \cite{balas} says that any acylindrically hyperbolic group
has a non elementary cobounded action on a quasi-tree. This quasi-tree, and
hence the orbits, will not be bounded or quasi-isometric to $\R$
and there are certainly such groups which have no unbounded isometric action
on any simplicial tree. Also the action of such a group on its Cayley graph
will have orbits which are coarse connected and proper, but not bounded
or quasi-isometric to $\R$. As for removal of the coarse connected condition,
Lemma \ref{ccnq} (ii) tells us that we will not be able to reduce such an
action to a cobounded quasi-action on any quasi-geodesic space.
We have already seen this in Example 3 in Section 5. Note that here the
orbits isometrically embed in a bounded valence tree and are themselves
proper spaces (as they are closed in this tree). Moreover the action is
metrically proper, so this example is as nice as it is possible to be
without having coarse connected orbits, yet the conclusion utterly fails.

\section{When quasi-orbits look like $\R$}

\subsection{Definitions}

In Theorem \ref{main} we took any quasi-action where the 
quasi-orbits were quasi-isometric to trees and could be quasi-isometrically
embedded in some proper metric space. We were able to reduce this
quasi-action
to a cobounded isometric action on a bounded valence tree, unless
the quasi-orbits were quasi-isometric to $\R$.
Now suppose that we have any quasi-action of an arbitrary group on
an arbitrary space where the quasi-orbits are quasi-isometric to $\R$.
This case is still covered by Theorem \ref{main}, giving us
a reduction to a cobounded quasi-action on $\R$. But as the quasi-orbits
are quasi-geodesic spaces, Lemma \ref{ccnq} (iii) tells us that we also have
some reduction to a cobounded isometric action on a hyperbolic space
which is quasi-isometric to $\R$. The point is that, with the lack of
an equivalent version of Theorem \ref{mswt} for these quasi-actions, we either
have to take the advantage of an isometric action or the space being
$\R$, but seemingly we cannot have both together.

We will deal with this case here, where we will see that there is
a genuine obstruction to reducing our quasi-action to an isometric action
on $\R$.
As an indication of what we should expect to see, a group $G$ acting
on $\R$ by orientation preserving isometries is exactly a homomorphism
from $G$ to $\R$. It was pointed out
in \cite{mnlms} that if instead we take a group $G$ and
a quasi-morphism $f:G\rightarrow\R$,
meaning that there is $D\geq 0$ 
such that for all $g_1,g_2\in G$ we have
$|f(g_1g_2)-f(g_1)-f(g_2)|\leq D$ (we call $D$ a defect of $f$)
then we obtain
a quasi-action of $G$ by translations on $\R$ by setting
$\alpha(g,x)=x+f(g)$. Indeed if we think of an orientation preserving
quasi-action of $\R$ as one where every map preserves each end of $\R$
then the orientation preserving $(1,0,C)$ quasi-actions of $G$ on $\R$
are precisely the quasi-morphisms $f:G\rightarrow\R$ of defect at
most $C$. We refer to this
example of a quasi-action as $G$ {\bf translating by} $f$ on $\R$. Note
that if $f$ is a genuine homomorphism
then this is an isometric action of $G$
on $\R$ and every orientation preserving isometric action of $G$ on $\R$
occurs in this way (and this action will be of type (3)$^+$ unless the
homomorphism is trivial).

Thus we begin this section by looking at isometric actions on hyperbolic
spaces $X$ where the orbits are quasi-isometric to $\R$. In fact
these are exactly
the hyperbolic actions of type (3) because an action of this type will
contain a loxodromic element and we can apply the Morse lemma: indeed
we will obtain a coarse equivariant version of this in Theorem \ref{twchr}
(for actions of type (3)$^+$) and Corollary \ref{rvrs} (for type (3)$^-$).
For the converse direction, orbits will not be bounded but are quasi-geodesic
spaces, so this action is not of type (1) nor of type (2) by Proposition
\ref{orblk}. However this orbit can only have two accumulation points on
the boundary $\partial X$ so the action cannot be of types (4) or (5).

So we assume that we have any type (3) action of an arbitrary group
$G$ on an arbitrary hyperbolic space $X$.
By definition $G$ fixes the subset $\{\zeta^+,\zeta^-\}$ of the boundary
$\partial X$ where the two limit points of the action
are $\zeta^\pm$ and therefore
$G$ either fixes or swaps the individual points. We now assume that $G$ is of
type (3)$^+$ (namely they are fixed)
because either $G$ or an index two subgroup will be of this form
and we will deal the case where they are swapped later in this section.

\subsection{Orientation preserving actions}
Now that we know $G$ fixes (say) $\zeta^+$ pointwise, we can bring in
Busemann functions. The relevant properties and results that we require
can be found in various places, for instance \cite{cdcmt} Section 3,
\cite{mnarth} Section 4 and \cite{bargen} Section 2 (however we warn that our
definition is actually minus their definition for reasons that will become clear
as we progress).

First of all, let $X$ be a hyperbolic space and let us take a
point $\zeta$ in the Gromov boundary $\partial X$ and a sequence
${\bf x}=(x_n)$ of points in $X$ tending to $\zeta$ in $X\cup\partial X$.
\begin{defn} (\cite{mnarth} Definition 4.3) \label{dbus}
  The {\bf quasi-horofunction} $\eta_{\bf x}:X\rightarrow\mathbb R$ of the
  sequence
  $\bf x$ is the  function defined by
\[\eta_{\bf x}(z)=\liminf_{n\rightarrow\infty} \Big(d_X(x_n,x_0)-d_X(x_n,z)
    \Big).\]
\end{defn}
Note that the modulus of the bracketed expression is bounded above
by $d_X(x_0,z)$ so that the lim inf is finite.
We then have the following:\\
$\bullet$
If ${\bf y}=(y_n)$ is another sequence tending to the same point $\zeta$
on $\partial X$ then \cite{mnarth} Lemma 4.6 tells us that
$\eta_{\bf y}$ is within bounded distance of $\eta_{\bf x}$ and the
bound depends only on the hyperbolicity constant of $X$ (using the Gromov
product definition) and $\eta_{\bf x}(y_0)$.\\
$\bullet$
If a group $G$ acts by isometries on $X$ fixing the point
$\zeta\in \partial X$ then the function $f_{\bf x}$
from $G$ to $\mathbb R$ given by $g\mapsto \eta_{\bf x}(gx_0)$
is a quasi-morphism (\cite{mnarth} Corollary 4.8).
Moreover using a different sequence tending to $\partial X$ (or a different
basepoint $x_0$) gives us a quasi-morphism which is within bounded distance
of $f_{\bf x}$. But given any quasi-morphism $f$, we can form the homogenisation
$\overline{f}$ of $f$ so that $\overline{f}(g^n)=n\overline{f}(g)$ for all
$n\in \Z$ and $g\in G$. This is the unique
homogeneous quasi-morphism in the equivalence class of all quasi-morphisms
within bounded distance of $f$; indeed we have $|f(g)-\overline{f}(g)|\leq D$
for any defect $D$ of $f$. Thus we define the {\bf Busemann quasi-morphism}
$B_\zeta$ of the action of $G$ on $X$ at $\zeta\in\partial X$ to be the
homogenisation of $f_{\bf x}$ and we see that $B$ is indeed independent of
the particular sequence $\bf x$ tending to $\zeta$.\\
$\bullet$
For such an action, an element $g\in G$ acts loxodromically if and only
if $B_\zeta(g)\neq 0$.

We note here two points: first that if our action of $G$
is on $\R$ itself (or restricts to a copy of $\R$)
and is orientation preserving then (by taking $x_n=n$) our definition
gives us that the Busemann quasi-morphism is simply the homomorphism
$g\mapsto g(0)$ (hence our use of minus the standard definition).
Second: we have to use genuine isometric actions to obtain a
Busemann quasi-morphism. For instance if we create a quasi-action
$\alpha$ by
taking the unit translation action of $\Z$ on $\R$ and conjugating by
the quasi-morphism $x\mapsto 2x\,(x\geq 0)$ and $x\mapsto x\,(x\leq 0)$
then the function $g\mapsto\alpha(g,0)$ (which would be the appropriate
formula to use by our first point, at least before we homogenise)
is readily seen not to be a quasi-morphism.
  
We will now give a complete description of type (3)$^+$ isometric
actions. We have seen above that if $G$ fixes a bi-infinite geodesic setwise
then the Busemann quasi-morphism is just the corresponding
homomorphism given by translation along the geodesic, so in general we would
expect that the Busemann quasi-morphism is given by (or is close to)
translation along an appropriate quasi-geodesic. As a type (3)$^+$ action
always has loxodromic elements, the obvious class of
quasi-geodesics (with domain $\Z$ here as opposed to $\R$) to take
would be those of the form
$n\mapsto l^n(x_0)$ where $l$ is any fixed loxodromic element of $G$
and $x_0\in X$ is any basepoint. It is then well known that there is
$K\geq 1$ such that for all $m,n$ we have
\[(1/K) |m-n|\leq d_X(l^m(x_0),l^n(x_0))=d(l^{m-n}(x_0),x_0)\leq K |m-n|.\]
This inequality holds for any loxodromic element acting on any metric
space, but we are in a hyperbolic space and so can take advantage of the
Morse lemma to get a reverse inequality.
\begin{lem} \label{mrse} 
  If $l$ is any loxodromic element acting on a hyperbolic space $X$ and
  $x_0$ is any point in $X$ then there
  is a constant $L\geq 0$ such that for any integers $M,i,N$ with $M<i<N$,
  we have
  \[d_X(l^N(x_0),l^M(x_0))-d_X(l^N(x_0),l^i(x_0))
    \geq d_X(l^i(x_0),l^M(x_0))-2L.\]
\end{lem}
\begin{proof}
  We have seen
  that the map from $[M,N]\cap\Z$ to $X$ given by $i\mapsto l^i(x_0)$ is
  a $(K,0)$ quasi-geodesic. The Morse lemma tells us that there is a
  constant $L$, depending only on $K$ and the hyperbolicity constant
  $\delta$ of $X$, such that on taking some geodesic $\gamma$ between the two
  points $l^M(x_0)$ and $l^N(x_0)$ in $X$, any $(K,0)$ quasi-geodesic with
  the same endpoints has Hausdorff distance at most $L$ from $\gamma$.
  Therefore if $i$ is between $M$ and $N$ then $l^i(x_0)$ lies on such
  a quasi-geodesic, giving us a point $t$ on $\gamma$ with
  $d_X(l^i(x_0),t)\leq L$. Thus we have
  \begin{eqnarray*}
    d_X(l^N(x_0),l^M(x_0))-d_X(l^N(x_0),l^i(x_0))&\geq& \\ 
    d_X(l^N(x_0),l^M(x_0))-d_X(l^N(x_0),t)-d(t,l^i(x_0))&\geq&\\
    d_X(l^N(x_0),l^M(x_0))-d_X(l^N(x_0),t)-L.&&
  \end{eqnarray*}
  But as $t$ lies on a geodesic between $l^M(x_0)$ and $l^N(x_0)$, we have
\[d_X(l^N(x_0),l^M(x_0))-d_X(l^N(x_0),t)=d_X(t,l^M(x_0))\mbox{ and thus }\]
   \begin{eqnarray*}
     d_X(l^N(x_0),l^M(x_0))-d_X(l^N(x_0),t)-L=d_X(t,l^M(x_0))-L&\geq&\\
d_X(l^i(x_0),l^M(x_0))-d_X(t,l^i(x_0))-L&\geq& \\     
     d_X(l^i(x_0),l^M(x_0))-2L.&&
\end{eqnarray*}                                  
\end{proof}
Here we use the term quasi-line for a geodesic metric space which is
quasi-isometric to $\R$.
\begin{thm} \label{twchr}
Suppose that $G$ is any group with an isometric action on an arbitrary
hyperbolic space $X$  which is of type (3)$^+$ and with Busemann
quasi-morphism $B$ at one of the two limit points.
Then this action can be reduced to the quasi-action $\beta$
which is translation by $B$ on $\R$. If the isometric action is cobounded
(that is, $X$ is a quasi-line) then this is an equivalence.
\end{thm}
\begin{proof}
  We merely require a coarse $G$-equivariant quasi-isometry from
  translation by $B$ on $\R$ to the given action restricted to
  $Orb(x_0)$, for $x_0$ an arbitrary point of $X$, as we can then
  compose with inclusion into $X$ to obtain a coarse $G$-equivariant
  quasi-isometric embedding from $\R$ to $X$ (which will be an
  equivalence if the action is cobounded). We will actually go in the
  other direction and show that we have a coarse $G$-equivariant
  quasi-isometry from $Orb(x_0)$ to $\R$. 

We first note that it does not matter whether we use translation
  by $B$ or by some inhomogeneous quasi-morphism within bounded distance
  of $B$ for our quasi-action on $\R$, because the identity on $\R$ will
  be a coarse $G$-equivariant isometry between these two quasi-actions.

  Take one of the two fixed points $p,q$ on
  $\partial X$, say $p$ (it will turn out not to matter which)
  and let $B$ be the Busemann quasi-morphism at $p$.
Take any basepoint $x_0$ and loxodromic element $l$ with attracting fixed
point $p$. For $n\geq 0$ we will define $\bf l$ to be the sequence
$(l^n(x_0))$ which tends to $p$.
We then define our map $F:Orb(x_0)\rightarrow\R$ by 
sending $g(x_0)$ to the appropriate value
$f_{\bf l}(g):=\eta_{\bf l}(gx_0)$ of the
quasi-horofunction $\eta_{\bf l}$. Note that $\eta_{\bf l}$
is well defined on $Orb(x_0)$
and $f_{\bf l}$ is a quasi-morphism within bounded distance of $B$.

On taking $g,h\in G$, we now need to consider the values of
$\eta_{\bf l}(gx_0)$ and $\eta_{\bf l}(hx_0)$ which by definition is
\[\liminf_{n\rightarrow\infty} \Big(d_X(l^n(x_0),x_0)-d_X(l^n(x_0),g(x_0))
  \Big)\]
for $\eta_{\bf l}(gx_0)$ and the obvious equivalent expression for
$\eta_{\bf l}(hx_0)$.

Note that if we have two
  bounded real valued sequences $(a_n)$ and $(b_n)$ with constants
  $c,d$ such that $c\leq a_n-b_n\leq d$ for sufficiently large $n$ then
  $c\leq\liminf_{n\rightarrow\infty} a_n - \liminf_{n\rightarrow\infty} b_n\leq d$.
  Therefore we estimate the expression
\[\Big(d_X(l^n(x_0),x_0)-d_X(l^n(x_0),g(x_0))
    \Big)-\Big(d_X(l^n(x_0),x_0)-d_X(l^n(x_0),h(x_0))
    \Big)\]
which is equal to $d_X(l^n(x_0),h(x_0))-d_X(l^n(x_0)-g(x_0))$.  

Now our bi-infinite $(K,0)$ quasi-geodesic $\sigma$ has limit points $p,q$
on $\partial X$. But any $g\in G$ acts as an isometry on $X$ and fixes
$p$ and $q$ because we have a type (3)$^+$ action. Thus $g(\sigma)$ is also
a $(K,0)$ quasi-geodesic between $p$ and $q$, thus again by the appropriate
Morse lemma (such as \cite{mnlms} Lemma 2.11, which specifically does not
assume the existence of a bi-infinite geodesic between $p$ and $q$)
we will have an integer $n(g)$, which in fact only depends on $g(x_0)$,
with $d_X(g(x_0),l^{n(g)}(x_0))\leq L$. We also
have an equivalent integer $n(h)$ for the element $h$.

Now take $n$ large enough that it is greater than $n(g)$ and $n(h)$.
Thus assuming that $n(h)\leq n(g)<n$ (if not then swap $g$ and $h$), 
we have by Lemma \ref{mrse} that
  \[d_X(l^n(x_0),l^{n(h)}(x_0))-d_X(l^n(x_0),l^{n(g)}(x_0))
    \geq d_X(l^{n(g)}(x_0),l^{n(h)}(x_0))-2L.\]
  But  $d_X(l^n(x_0),h(x_0))\geq  d_X(l^n(x_0),l^{n(h)}(x_0))-L$
  and $d_X(l^n(x_0),l^{n(g)}(x_0))\leq d_X(l^n(x_0),g(x_0))+L$
so
\[d_X(l^n(x_0),h(x_0))-d_X(l^n(x_0),g(x_0))
    \geq d_X(g(x_0),h(x_0))-6L\]  
for all large $n$ and so the same holds for our difference of
quasi-horofunctions $\eta_{\bf x}(g(x_0))-\eta_{\bf x}(h(x_0))$ by taking
lim infs. But here we can replace this difference with
the modulus of this difference (by considering when $d_X(g(x_0),h(x_0))$
is at least or at most $6L$ separately). 

Of course we also have
  \[d_X(h(x_0),g(x_0))\geq d_X(l^n(x_0),h(x_0))-d_X(l^n(x_0),g(x_0))
\geq  -d_X(h(x_0),g(x_0))\]
which again will hold for $\eta_{\bf x}(g(x_0))-\eta_{\bf x}(h(x_0))$
by taking lim infs. Thus we obtain
\[d_X(h(x_0),g(x_0))-6L\leq |\eta_{\bf x}(g(x_0))-\eta_{\bf x}(h(x_0))|
\leq d_X(h(x_0),g(x_0)),\] so that our map $F$ sending
$g(x_0)$ to $f_{\bf l}(g)=\eta_{\bf l}(gx_0)$ is indeed a quasi-isometric
embedding.
Moreover $F$ is coarse onto because its
 image is coarse connected (as $Orb(x_0)$ is) and unbounded both above
 and below, as $G$ contains a loxodromic element.
 It is also coarse $G$-equivariant because
 \[d_\R(Fg(h(x_0)),gF(h(x_0)))=|f_{\bf l}(gh)-g(f_{\bf l}(h))|
   =|f_{\bf l}(gh)-(f_{\bf l}(h)+f_{\bf l}(g))|\] which is bounded as $f_{\bf l}$
 is a quasi-morphism.
\end{proof}  
\begin{co} \label{symb}
  If $G$ has a type (3)$^+$ action on some hyperbolic space $X$
  with fixed points $\zeta^\pm$ on $\partial X$
  then the respective homogeneous Busemann quasi-morphisms $B^+$ and
  $B^-$ of this action are minus each other.
\end{co}
\begin{proof}
  As in the proof of Theorem \ref{twchr} with $p$ set to be $\zeta^+$ and
  $q=\zeta^-$, we let the sequence ${\bf m}$ be $(l^{-n}(x_0))$ which tends
  to $\zeta^-$.
  We now run through this above proof with the sequence $\bf l$ and 
  setting $h=id$. On taking any $g\in G$,
  first suppose that the integer $n(g)$ obtained in this proof is at least
  zero. We conclude that
  \[d(g(x_0),x_0)-6L\leq \eta_{\bf l}(gx_0)\leq d(g(x_0),x_0).\]
  But we can also run through this proof with the same $g\in G$ but
  replacing $l$ with $l^{-1}$, whereupon
we are now considering $\eta_{\bf m}(gx_0)$ and the previous integer $n(g)$
will now be $-n(g)$. Now our inequalities become
\[d(g(x_0),x_0)-6L\leq -\eta_{\bf m}(gx_0)\leq d(g(x_0),x_0).\]
Thus on reversing this and adding the two together, we obtain
\[-6L\leq \eta_{\bf l}(gx_0)+\eta_{\bf m}(gx_0)\leq 6L.\]
For elements $g\in G$ with $n(g)\leq 0$, the two inequalities are reversed
but the sum is the same. However the functions
$g\mapsto\eta_{\bf l}(gx_0)$ and $g\mapsto\eta_{\bf m}(gx_0)$ from $G$ to $\R$
are within bounded distance of $B^+$ and $B^-$ respectively, so $B^++B^-$ is a
homogeneous quasi-morphism which is bounded, thus is zero.
\end{proof}
\subsection{Orientation reversing actions}
Suppose a group $G$ has an isometric action on some hyperbolic space $X$
which is of type (3)$^-$. Then
$G$ can be decomposed as $G^+\cup t G^+$  where $G^+$ is the index two
orientation preserving subgroup, thus $G^+$ has a type (3)$^+$
action on our space $X$, and $t$ is any fixed element of $G$ that swaps the
two limit points. Here we will define the {\bf Busemann quasi-morphism} of the
action to be (either of) the (two) Busemann quasi-morphism(s) of $G^+$.
Note that all elements of the coset $tG^+$ must be elliptic:
if not then such an element fixes something on the boundary, so its
square fixes at least three points on the boundary and hence is not
parabolic or loxodromic. In this case we would expect an equivalent statement
to Theorem \ref{twchr} but with some sort of quasi-action on $\R$ that
is like a isometric dihedral action on $\R$ rather than with just
translations. Indeed in analogy with isometric dihedral actions, suppose
we take $H$ to be any index 2 subgroup of $G$ and $q$ to be any homogeneous
quasi-morphism defined on $H$. Then for any element $h\in H$, not only is
$x\mapsto x+q(h)$ an isometry of $\R$ but so is $x\mapsto -x-q(h)$.
On taking any element $t\in G\setminus H$ and setting
$\beta:G\times \R\rightarrow \R$ to be
\begin{eqnarray*}
\beta(h,x)&=&x+q(h)\\
  \mbox{ and } \beta(th,x)&=&-x-q(h)\mbox{ for }x\in\R\mbox{ and }h\in H
\end{eqnarray*}
we might hope that $\beta$ is a quasi-action of $G$ on $\R$. It can be
checked that this is true if $q$ satisfies the antisymmetry condition
that $q(tht^{-1})=-q(h)$.
For instance on taking $h_1,h_2\in H$ we have
  \[\beta(h_1,\beta(th_2,x))=\beta(h_1,-x-q(h_2))=-x+q(h_1)-q(h_2)\]
  whereas
  \[  \beta(h_1th_2,x)=\beta(t\cdot t^{-1}h_1th_2,x)=-x-q(t^{-1}h_1th_2)\]
  because $t^{-1}h_1th_2\in H$. This also implies that $q(t^{-1}h_1th_2)$
  is close to $q(t^{-1}h_1t)+q(h_2)$ but $q(t^{-1}h_1t)=-q(h_1)$ by antisymmetry,
so $q(h_1,\beta(th_2,x))$ is close to
$\beta(h_1th_2,x)$, independently of $x$ and $h_1,h_2$.

Note that if the antisymmetry condition holds under conjugation by
$t$ then it will also hold under conjugation by any element 
$s=th$ in $G\setminus H$ because a homogeneous quasi-morphism $q$ is constant
on conjugacy classes. We will refer to this as $q$ is antisymmetric in $G$.
Moreover we will have $q(s^2)=0$ because 
$q(s^2)=q(s\cdot s^2\cdot s^{-1})=-q(s^2)$.
Note also that again $\beta$ is an isometric action if $q$ is
a genuine homomorphism from $H$ to $\R$ which is antisymmetric in $G$.
Moreover every orientation reversing isometric action of $G$ on $\R$
occurs in this way and this action will be of type (3)$^-$ unless the
homomorphism is trivial. We call the quasi-action $\beta$ given by
$q$ the {\bf dihedral quasi-action} of $G$ on $\R$ by $q$. Note that a dihedral
quasi-action does not really depend on which $t$ in $G\setminus H$
we choose (this just identifies the origin of $\R$ with the fixed
point of $t$).

We can now give our equivalent statement for type (3)$^-$ actions
to Theorem \ref{twchr}.
\begin{co} \label{rvrs}
  Suppose that $G$ is any group with an isometric action of type (3)$^-$
  on an arbitrary hyperbolic space $X$. Let $G^+$ be the orientation
  preserving subgroup of $G$ and $B$ the Busemann quasi-morphism of $G$.
  Then our isometric action of $G$ can be reduced to the quasi-action
  $\beta$ which is the dihedral quasi-action of $G$ on $\R$ given by $B$.
If the isometric action is cobounded
(that is, $X$ is a quasi-line) then this is an equivalence.
\end{co}
\begin{proof}
  We take a loxodromic element $l$ of $G$, which will necessarily lie in
  $H$, with sequences $\bf l$ and $\bf m$ defined as before.
  We must now show that the Busemann quasi-morphism $B$ from $G^+$ to $\R$
  based at the attracting fixed point $\zeta^+$
  of $l$, namely the limit of $\bf l$, is antisymmetric in $G$. We pick
  any $t\in G\setminus G^+$ and we check that $B(tht^{-1})=-B(h)$ for
  any $h\in G^+$. (See also \cite{bargen} Proposition 2.9 which
  establishes this using quasicocycles.)
  This requires considering $\eta_{\bf l}(tht^{-1}x_0)$ which is the lim inf of
  $d_X(x_n,x_0)-d_X(tht^{-1}(x_0))$ for our basepoint $x_0$ and $x_n=l^n(x_0)$.
  Let us set $y_n=t^{-1}(x_n)$ for $n\in\Z$, so that the sequence
  ${\bf y}:=(y_n)$ tends to the
  repelling fixed point $\zeta^-$ of $l$ as $n$ tends to infinity.
  We then have that
  \[d_X(x_n,x_0)-d_X(x_n,tht^{-1}(x_0))=d_X(y_n,y_0)-d_X(y_n,h(y_0))\]
  and taking lim infs gives us that $\eta_{\bf l}(tht^{-1}x_0)=\eta_{\bf y}(hy_0)$.
  Now the function $h\mapsto \eta_{\bf l}(hx_0)$ is a quasi-morphism on $G^+$
  which is within bounded distance to $B$, whereas $h\mapsto \eta_{\bf y}(hy_0)$
  is within bounded distance to the Busemann quasi-morphism at $\zeta^-$
  which is $-B$ by Corollary \ref{symb}. As $h$ and $tht^{-1}$ are both in $G^+$,
  this gives us that $B(tht^{-1})=-B(h)$ and so $\beta$ really is a
  quasi-action.
  
As for showing that our isometric action reduces to this quasi-action $\beta$,
it is again enough to replace the quasi-action $\beta$ on $\R$ with
the quasi-action $\gamma$ which replaces $B$ with the quasi-morphism
$g\mapsto \eta_{\bf l}(gx_0)$ from $G^+$ to $\R$, because this is within
bounded distance of $B$. Thus
we now extend the proof of Theorem \ref{twchr} by letting our map
  $F$ be defined on the orbit of $x_0$ under $G$, not just $G^+$, by
  sending $g(x_0)$ to $\eta_{\bf l}(gx_0)\in\R$. This is still well defined,
  although $\eta$ need not be a quasi-morphism on $G$ anymore.

  Moreover in running through the proof of Theorem \ref{twchr}, the Morse lemma 
  is still valid for elements $g\in G\setminus G^+$ because $g$ sends our
  $(K,0)$ quasi-geodesic to another such quasi-geodesic with the same two
  endpoints, even if they are swapped. Thus the estimates that showed
  $F$ was a quasi-isometry still apply here.

  We do need however to show that $F$ is a coarse $G$-equivariant map from
  $Orb_G(x_0)$ to $\R$ and not just coarse $G^+$-equivariant.
    To this end, when considering
  $\eta_{\bf l}(gx_0)$ we cannot assume that $g\mapsto \eta_{\bf x}(gx_0)$
  is a quasi-morphism on $G$. However the argument above which
  established that $\eta_{\bf l}(tht^{-1}x_0)=\eta_{\bf y}(hy_0)$ for $h\in G^+$,
  or even $h\in G$, also gives us that $\eta_{\bf l}(thx_0)$
  is within $d_X(x_0,y_0)$ of $\eta_{\bf y}(hy_0)$ for $h\in G$ too.

  We also have that $g\mapsto \eta_{\bf l}(gx_0)$ and
  $g\mapsto \eta_{\bf y}(gy_0)$ are quasi-morphisms on $G^+$ (though not on $G$).
Thus for $h_1,h_2\in G^+$ we have
\[
|Fth_1(h_2x_0)-\gamma(th_1,F(h_2x_0))|=
|\eta_{\bf l}(th_1h_2x_0)+\eta_{\bf l}(h_2x_0)+\eta_{\bf l}(h_1x_0)|.\]
But from the above we have that $\eta_{\bf l}(th_1h_2x_0)$ is within bounded
distance of $\eta_{\bf y}(h_1h_2y_0)$, hence also of $-B(h_1h_2)$ by Corollary
\ref{symb} and thus $-B(h_1)-B(h_2)$. Meanwhile
$\eta_{\bf l}(h_ix_0)$ is within bounded distance of $B(h_i)$ for $i=1,2$, so
that our expression is bounded independently of the group elements.
Similarly we also have
\[
|Fh_1(th_2x_0)-\gamma(h_1,F(th_2x_0))|=
|\eta_{\bf l}(h_1th_2x_0)-\eta_{\bf l}(th_2x_0)-\eta_{\bf l}(h_1x_0)|,\]
with $\eta_{\bf l}(tt^{-1}h_1th_2x_0)$ within bounded
distance of $\eta_{\bf y}(t^{-1}h_1t\cdot h_2y_0)$ and hence also of
$\eta_{\bf y}(t^{-1}h_1ty_0)+\eta_{\bf y}(h_2y_0)=\eta_{\bf l}(h_1x_0)
+\eta_{\bf y}(h_2y_0)$.
But $\eta_{\bf l}(th_2x_0)$ is also close to $\eta_{\bf y}(h_2y_0)$.

Finally we have
\[
|Fth_1(th_2x_0)-\gamma(th_1,F(th_2x_0))|=
|\eta_{\bf l}(th_1th_2x_0)+\eta_{\bf l}(th_2x_0)+\eta_{\bf l}(h_1x_0)|.\]
Again $\eta_{\bf l}(th_1th_2x_0)$ is within bounded
distance of $\eta_{\bf y}(h_1th_2y_0)$, which from the previous case
(but swapping $\bf l$ and $\bf y$) is
close to $\eta_{\bf y}(h_1y_0)+\eta_{\bf l}(h_2x_0)$ and so to
$-B(h_1)+B(h_2)$. But $\eta_{\bf l}(th_2x_0)+\eta_{\bf l}(h_1x_0)$ is close to
$\eta_{\bf y}(h_2y_0)+B(h_1)$ and so to
$-B(h_2)+B(h_1)$, thus we have covered all cases.
\end{proof}
\subsection{Distinguishing isometric actions}  

Returning now to quasi-actions, suppose that we have a quasi-action $\alpha$
of a group $G$ on a metric space where quasi-orbits are quasi-isometric
to $\R$, so
that we can reduce it to a cobounded isometric action
on some quasi-line. We now consider 
when $\alpha$ can be reduced to a cobounded isometric action on $\R$.
Note that if this action is of type (3)$^+$ (respectively (3)$^-$)
then any reduction of $\alpha$ to an action on some hyperbolic space
will also be of type (3)$^+$ (respectively (3)$^-$). Theorem \ref{twchr}
and Corollary \ref{rvrs} tell us it is sufficient that the
corresponding Busemann function of this reduced isometric action is a genuine
homomorphism to $\R$.  This is because if we have a reduction
of $\alpha$ to an action $\gamma$ on a quasi-line
where the Busemann quasi-morphism $B$
is genuinely a homomorphism $\theta$ then these results applied to
$\gamma$ say that the corresponding translation/dihedral
quasi-action on $\R$ given by the Busemann quasi-morphism $B$ will be an
isometric action that is a reduction of $\gamma$, therefore also a
reduction of $\alpha$.

Conversely if a reduction of $\alpha$ to an 
isometric action on $\R$ exists, this reduced action will be by
translations or will be dihedral with $g\mapsto g(0)$ the Busemann
quasi-morphism. So for this reduction of $\alpha$, the
Busemann quasi-morphism will be a homomorphism of $G$ to $\R$.
However this answer misses two important points: first, do we obtain the
same or a similar Busemann quasi-morphism whenever we reduce $\alpha$ to an
isometric action on some hyperbolic space? Second, even if we have a positive
answer to the above question, how do we characterise
which quasi-actions can be made into isometric actions on $\R$ without
first having to turn them into an isometric action on a particular hyperbolic
space (which might not be an easy process)? We deal with the second point
first, where our answer depends directly on our given quasi-action.
\begin{thm} \label{dpell}
  Let $\alpha$ be a quasi-action of a group $G$ on a metric space which can
  be reduced to an isometric action of type (3)$^+$ on some hyperbolic space.\\
$\bullet$ If the set of elliptic elements of $\alpha$ is equal to the kernel
  of a homomorphism from $G$ to $\R$ then there is a reduction of $\alpha$
  to an isometric translation action of $G$ on $\R$.\\
$\bullet$ If the set of elliptic elements of $\alpha$ is not equal to the
  kernel of any homomorphism from $G$ to $\R$ then $\alpha$ cannot be
  reduced to any isometric action on a hyperbolic space that is proper
  or is CAT(0).\\
  Let $\alpha$ be a quasi-action of a group $G$ on a metric space which can
  be reduced to an isometric action of type (3)$^-$ on some hyperbolic space.\\
  $\bullet$ If there is an index 2 subgroup $H$ of $G$ and a
  homomorphism from $H$ to $\R$ such that the set of elliptic
  elements of the quasi-action $\alpha$ consists of the union of this kernel
  and $G\setminus H$ then there is a reduction of $\alpha$
  to a isometric dihedral action of $G$ on $\R$.\\
  $\bullet$ If there is no index 2 subgroup $H$ of $G$ having a
  homomorphism to $\R$ where the set of elliptic
  elements of the quasi-action $\alpha$ consists of the union of this kernel
  and $G\setminus H$ then $\alpha$ cannot be
  reduced to any isometric action on a  hyperbolic space that is proper
  or is CAT(0).  
\end{thm}
\begin{proof}
  First suppose that we are in the type (3)$^+$ case.
Say there is no homomorphism
$\theta:G\rightarrow\R$ with $ker(\theta)$ equal to the set of elliptic
elements of $\alpha$. Elements of $G$
have the same type under any reduction, with type (3) actions
having no parabolic elements. Thus if there were a reduction to an isometric
translation action on $\R$ then this Busemann quasi-morphism would just be the
homomorphism obtained by this translation action, with the kernel equal to
the set of elliptic elements of this action and so of $\alpha$ too.
Furthermore in \cite{cdcmt} Corollary 3.9 it is shown that
for any action on a proper hyperbolic metric space with a fixed point on the
boundary, the Busemann quasi-morphism so obtained is a
homomorphism. It is also mentioned there that this is true for any hyperbolic
space (not necessarily proper) that is CAT(0).

Now suppose the set of elliptic elements of $\alpha$ agrees exactly
with the kernel of a homomorphism $\theta:G\rightarrow\R$. We reduce
$\alpha$ to any isometric action on a hyperbolic space, which
will be of type (3)$^+$. We take the
Busemann quasi-morphism $B$ of this action, which will be homogeneous
and whose zero set is exactly $ker(\theta)$, because the elliptic
and non elliptic elements are unchanged under reduction. In particular
we have $B(ghg^{-1}h^{-1})=0$ for every commutator of $G$ because these
elements all lie in the kernel. But by Barvard (\cite{bvd} Lemma 3.6)
we have for $B$ a homogeneous quasi-morphism that the supremum of
$|B(ghg^{-1}h^{-1})|$ over $g,h\in G$ is the defect of $B$, so here it is zero
implying that $B$ is a homomorphism. So on now applying Theorem \ref{twchr},
the resulting quasi-action on $\R$ is an isometric action.

Finally say that the quasi-action can be reduced to some
type (3)$^-$ action on a hyperbolic space. The first part of the
proof above goes through as well for this case by setting $H$ in the
statement of our theorem to be
the index two orientation preserving subgroup $G^+$ of the action.
This also works for the second part
by applying Corollary \ref{rvrs} instead of Theorem \ref{twchr},
except that we need to show that $G^+$ is equal to the given index 2 subgroup
$H$. If not then we can take some element $h$ in $H\setminus G^+$. As all
elements in $G\setminus G^+$ are elliptic in a type (3)$^-$ action, we must
have $\theta(h)=0$ for $\theta$ our given homomorphism from $H$ to $\R$.
But on taking some $h_1\in H$ with $\theta(h_1)=\theta(hh_1)\neq 0$
(such an $h_1$ exists because there are loxodromic elements in $G$ and their
square will lie in $H$),
we have that $h_1$ and $hh_1$ are not elliptic so they must lie in $G^+$.
Thus $hh_1h_1^{-1}$ is in $G^+$ which is a contradiction.
\end{proof}
\begin{ex}
On taking the free group $F_2$ and a quasi-morphism $q$ which is not
within bounded distance of any homomorphism, say a Brooks quasi-morphism,
we have that the seemingly innocuous translation
quasi-action $\alpha$ of $F_2$ on $\R$ given by $\alpha(g,x)=x+q(g)$
cannot be reduced to any isometric action on a CAT(0) space or on a proper
hyperbolic space. (Indeed it cannot be reduced to an isometric action
on any proper metric space, say by \cite{mrgo} Proposition 4.5 and Theorem
4.38.)
\end{ex}

We would now like to deal with our first point: if we turn a quasi-action
of some group $G$
into an isometric action of type (3) then do we always get the same
Busemann quasi-morphism $B$
(once we have homogenised)? Of course we can always rescale the metric, so
that $B$ and $\lambda B$ can occur for any $\lambda\neq 0$. However
\cite{abos} Corollary 4.16 gives plenty of examples of
homogeneous quasi-morphisms of a group where the zero
set is just the identity. Indeed this can happen with homomorphisms, for
instance take $G=\Z^2=\langle x,y\rangle$ and homomorphisms
$\theta_\pm:G\rightarrow\R$ given by $\theta_\pm(x)=1$,
$\theta_\pm(y)=\pm\sqrt{2}$. Then $\theta_+$ and $\theta_-$ do not look
equivalent but both have zero kernel.

It turns out that scaling is the only ambiguity.
\begin{prop} \label{smqm}
  Suppose that $\alpha$ is a quasi-action of a group $G$ on a metric space
  which can be reduced to an isometric action
  of type (3) on a hyperbolic space. If $\beta$ and $\gamma$  are two
  such reductions (not necessarily to the same hyperbolic space) then
  the homogeneous Busemann quasi-morphisms
  $B_\beta,B_\gamma$ are related by scaling.
\end{prop}
\begin{proof}
First assume that $\beta$ is of type (3)$^+$, in which case so is $\gamma$.
  Applying Theorem \ref{twchr} to both $\beta$ and $\gamma$, each can be reduced
to the translation quasi-actions on $\R$ by the Busemann quasi-morphisms
$B_\beta$ and $B_\gamma$ respectively, which we will call $\beta'$ and
$\gamma'$. As these quasi-actions on $\R$ are
both cobounded (unless one quasi-morphism is zero, in which case
$\alpha$ is bounded and so both are), they are equivalent by Proposition
\ref{mnrd} (ii). Thus we have a $(K,\epsilon,C)$
quasi-isometry $F:\R\rightarrow \R$ which
is $G$-coarse equivariant from $\beta'$ to $\gamma'$. The latter point
means we have $M\geq 0$ such that for all $g\in G$ and $x\in\R$,
\[d_\R (F\beta'(g,x),\gamma'(g,Fx))\leq M\mbox{ so }
|F(x)+B_\gamma(g)-F(x+B_\beta(g))|\leq M.\]  
Putting $x=0$ in this equation allows us to establish our key property,
which is that if we have a subset $S$ of $G$ such that
the values of $|B_\beta(g)|$ over $g\in S$ form 
a bounded subset of $\R$
then $|F(0)+B_\gamma(g)-F(B_\beta(g))|\leq M$ implies
\[|B_\gamma(g)|\leq M+K|B_\beta(g)|+\epsilon\]
and so the values of $|B_\gamma(g)|$ over the same subset $S$ are bounded too.

Now take any $g\in G$ with $B_\beta(g)\neq 0$ (so that $B_\gamma(g)\neq 0$
either as equivalent quasi-actions have the same loxodromic elements).
We rescale $B_\beta,B_\gamma$ so that both have value 1 on $g$. For any $h\in G$
and $n\in N$, we choose $m_n\in\Z$ to be the integer below $B_\beta(h^n)$,
so that we have
\[m_n\leq B_\beta(h^n)<m_n+1.\]
As $B_\beta$ and $B_\gamma$ are both homogeneous quasi-morphisms, we obtain
\[|B_\beta(g^{-m_n}h^n)+m_n-nB_\beta(h)|\leq D_\beta\]
and the equivalent equation for $\gamma$. As $|m_n-nB_\beta(h)|\leq 1$
by construction, we have that $|B_\beta(g^{-m_n}h^n)|$ is bounded over $n\in\N$.
Therefore so is  $|B_\gamma(g^{-m_n}h^n)|$ by our point above, which implies
that $|m_n-nB_\gamma(h)|$ is bounded too. Combining this with
$|m_n-nB_\beta(h)|\leq 1$ tells us that $n |B_\beta(h)-B_\gamma(h)|$ is also
bounded and hence $B_\beta(h)=B_\gamma(h)$, implying that $B_\beta=B_\gamma$ after
rescaling.

As for the case where $\beta$ is an action of type of (3)$^-$, the Busemann
quasi-morphism of this action is by definition the Busemann quasi-morphism
of the index 2 orientation
preserving subgroup $G^+$, so we restrict $\beta$ to $G^+$ whereupon it
becomes a type (3)$^+$ action. We do the same for $\gamma$ and its orientation
preserving subgroup $G^*$, but the argument at the end of Theorem \ref{dpell}
works to show that $G^+$ and $G^*$ are the same index 2 subgroup even if
$\theta$ is a homogeneous quasi-morphism, provided our loxodromic element
$h_1$ has $\theta(h_1)$ big enough, so we can take $h_1$ to be a large power
of itself if necessary. We now restrict $\beta$ and $\gamma$ to $G^+$,
whereupon they are still both reductions of the same quasi-action,
namely $\alpha$ restricted to $G^+$.
Now we can apply the first part of the proof to these restricted actions.
\end{proof}

Note that for two translation (respectively dihedral) quasi-actions on $\R$, if
their Busemann quasi-morphisms are related by scaling then they are
equivalent quasi-actions. This then gives us a complete and concrete description
of all type (3)$^+$ and type (3)$^-$ cobounded actions, up to equivalence,
of a given group $G$ over all hyperbolic metric spaces: for type (3)$^+$
actions, it is the projective vector space of homogeneous quasi-morphisms
of $G$. For type (3)$^-$ actions it is the disjoint union over each index
2 subgroup $H$ of the projective vector space of homogeneous
quasi-morphisms of $H$ which are antisymmetric in $G$.

\subsection{Simplicial actions}
In this paper we have been interested in when
quasi-actions of a group $G$ on a hyperbolic space $X$ can be reduced
to genuine actions on a simplicial tree. These actions have always been
isometric actions but we might want to ask whether they can be taken
to be simplicial actions. This is only relevant to type (3) actions
because for all trees other than the simplicial line $L$ (namely $\R$
with the simplicial structure where the vertices are $\Z$), any isometry
is a simplicial automorphism. However for $L$ the simplicial automorphisms
form a proper subgroup of of the isometry group and it could be for
some applications we only want to allow the former. Therefore we give
here the equivalent version of Theorem \ref{dpell}, where the obvious
condition of taking homomorphisms to $\Z$ rather than to $\R$
does indeed hold. 
\begin{co}
  Given the same hypotheses as Theorem \ref{dpell}, we have that
  a reduction of $\alpha$ to a simplicial translation/dihedral action
  of $G$ on the simplicial line $L$ exists or does not exist exactly
  under the same conditions as Theorem \ref{dpell}, other than that all
  homomorphisms from $G$ or an index 2 subgroup are to $\Z$ rather than
  to $\R$.
\end{co}
\begin{proof} Here we take both type (3)$^+$ and type (3)$^-$ actions together.
  The proof of the existence of such a homomorphism being a necessary
  condition is exactly as in the
  proof of Theorem \ref{dpell}, other than on changing $\R$ to $\Z$.

  As for sufficiency, the proof of Theorem \ref{dpell} again gives us that our
  Busemann quasi-morphism $B$ is a genuine homomorphism from $G$ to $\R$,
  so that Theorem \ref{twchr} or Corollary \ref{rvrs} gives us an
  isometric translation/dihedral action of $G$ on $\R$ by $B$.
  But moreover the kernel of $B$ agrees with the kernel of some homomorphism
  $\theta$ from $G$ (or from $G^+$) 
  to $\Z$. This means that $B$ is just a rescaling
  of $\theta$ and so we can rescale this isometric action of $G$ on $\R$
  to make the Busemann quasi-morphism equal to $\theta$, whereupon the
  action on $L$ is now simplicial.
\end{proof}

\section{Applications}

\subsection{Splittings of groups from quasi-actions}
We finish by applying our results to some general situations, with an
emphasis on finitely generated groups. Indeed
by results in \cite{ser} we know that a finitely generated group has
an unbounded action
on a tree (equivalently has an action without a
global fixed point) if and only if it splits as an amalgamated
free product or HNN extension. In this statement
the action is assumed to be simplicial
(we can subdivide all edges if any are inverted).
Thus Theorem \ref{main} now becomes:
\begin{thm} \label{summ}
  Let $G$ be any finitely generated group having an unbounded quasi-action
  on any metric space where the quasi-orbits both quasi-isometrically embed
  in some tree and also into some proper metric space. Either the quasi-orbits
  are quasi-isometric to $\R$ or $G$ splits non trivially as a finite graph
  of groups, where the image of any edge group under inclusion into a vertex
  group has finite index.
\end{thm}
\begin{proof}
  By Proposition \ref{fgn} (i) the quasi-orbits are coarse connected, so
  that Theorem \ref{main} applies. If the quasi-orbits are not quasi-isometric
  to $\R$ then we must be in the case where the quasi-action reduces
  to a cobounded
  isometric action on a bounded valence bushy tree. As the original
  quasi-action is unbounded, so is this action. Moreover a bushy tree is
  not the simplicial line, so this action is also by simplicial automorphisms
  and thus Serre's result gives us a splitting. As the tree is of bounded
  valence and the action is cobounded, the quotient graph is finite.
  The statement about edge groups is because the tree is of bounded valence.
\end{proof}
\begin{co} \label{amn}
Suppose that a finitely generated group   
$G$ has the property that every quasi-morphism of $G$ and of its
index 2 subgroups is within
bounded distance of a homomorphism (for instance amenable groups).
If $G$ has an unbounded quasi-action
  on any metric space where the quasi-orbits both quasi-isometrically embed
  in some tree and also into some proper metric space then $G$ splits non
  trivially as a finite graph
  of groups, where the image of any edge group under inclusion into a vertex
  group has finite index.
\end{co}
\begin{proof}
This follows by Theorem \ref{summ} unless the quasi-orbits are quasi-isometric
to $\R$. If so then by Theorem \ref{twchr} or Corollary \ref{rvrs} we have that
the quasi-action reduces to a translation/dihedral quasi-action with
respect to the relevant Busemann quasi-morphism $B$ of $G$ or an index 2
subgroup. But Busemann functions are homogeneous, thus $B$ must be a non
trivial homomorphism to $\R$ by the hypothesis on $G$.
Now a finitely generated group with such a
homomorphism to $\R$ will also have a surjective homomorphism $\theta$ say
to $\Z$. Thus in the orientation preserving case $G$ has a non trivial
simplicial action on the line. In the orientation reversing case with
$G^+$ the index 2 orientation preserving subgroup of $G$ and $t$ any fixed
element of $G\setminus G^+$ as before, we have our non trivial homomorphism
$B$ from $G^+$ to $\R$ which we know satisfies $B(tgt^{-1})=-B(g)$ for all
$g\in G^+$. As the image $B(G^+)$ is a copy of $\Z^n$ in $\R$ for some
$n\geq 1$, we can compose with a surjective homomorphism $\chi$ from
$B(G^+)$ to $\Z$, whereupon $\chi B$ is a homomorphism from $G^+$ onto $\Z$
which is also antisymmetric in $G$. Then we have seen that sending
$g$ in $G^+$ to $x\mapsto x+\chi B(g)$ and $tg$ to $-x-\chi B(g)$ is a
homomorphism of $G$ to $Isom(\R)$. But as the image of $\chi B$ is $\Z$,
this is also a homomorphism to the dihedral group $Isom(\Z)$ which is onto.
Thus $G$ has a non trivial simplicial action on the line in this case too.
\end{proof}

We now convert this into a statement about actions of groups on proper
quasi-trees. This differs vastly from the case without the finiteness
condition of being proper, which was considered in \cite{balas}.
\begin{co} \label{qttt}
  Suppose that $G$ is any finitely generated group with an unbounded
  isometric action on any proper quasi-tree $X$.
Then $G$ acts by automorphisms
and without a global fixed point on some simplicial tree with bounded valence.
Thus $G$ has the structure of a fundamental group of a (non trivial) finite
graph of groups where all edge groups have finite index in all of the
respective vertex groups.
\end{co}
\begin{proof}
  This is again Theorem \ref{summ} unless quasi-orbits are
  quasi-isometric to $\R$. 
  If so then we can also apply Theorem \ref{twchr} or Corollary \ref{rvrs} to
  obtain a reduction of this action to a
translation/dihedral quasi-action of $G$ on $\R$ with
respect to the relevant Busemann quasi-morphism $B$ of $G$ or an index 2
subgroup. But $B$ is obtained as the Busemann quasi-morphism of the action
of $G$ on $X$ at either of the two limit points in $\partial X$ of an
orbit. This is a homomorphism if $X$ is a proper hyperbolic space
by \cite{cdcmt} Corollary 3.9. We can now follow the proof of
Corollary \ref{amn}.
\end{proof}
{\bf Example 1}:
We cannot replace $X$ is proper with $X$ is quasi-isometric to
  some proper space. To see this, take a hyperbolic group $H$ with property (T)
  and a homogeneous quasi-morphism $q$ of $H$ which is not a homomorphism
  (ie not the trivial homomorphism).
By \cite{abos} Lemma 4.15, we can obtain from $q$ an infinite generating
set $S$ for $H$ such that the Cayley graph $Cay(H,S)$ is quasi-isometric to
$\R$ (and $q$ is indeed the Busemann quasi-morphism of the isometric action
of $H$ on $Cay(H,S)$ by Proposition \ref{smqm}).\\
\hfill\\
{\bf Example 2}: Just because our finitely generated group $G$ splits
with respect to an action on a bounded valance tree, this does not imply that
any edge or vertex group will itself be finitely generated. Note that
these subgroups are commensurable in $G$, so if one is finitely
generated then they all are. But let $G$ be a closed orientable 
surface group of genus at least 2. Then $G$ surjects to $F_2$, so certainly
has a splitting where each edge group embeds with finite index
in its respective vertex groups. Note that any vertex group is commensurable
to any of its conjugates. But $G$ is a torsion free hyperbolic 
group where any finitely generated subgroup is quasi-convex. In
\cite{cnmh} Theorem 3.15
it is shown that an infinite quasi-convex subgroup of a hyperbolic
group has finite index in its commensurator. Hence
if such an action on a bounded valence tree existed with finitely generated
vertex/edge stabilisers then these would either have finite index in $G$,
so that actually the splitting is trivial, or the stabilisers are themselves
trivial in which case $G$ would be a free group.\\
\hfill\\
{\bf Example 3}: To give an example of a finitely presented group that
has unbounded actions on trees but no unbounded actions on locally finite
trees, take
a finitely presented group $G$ which has property (FA) so it has no unbounded
actions on any tree. Suppose that $G$ has a finitely generated subgroup $H$ of
infinite index in $G$ with the following
property: if $H$ is contained in a finite index subgroup $L$ of $G$ then
$L=G$. (This is sometimes known in the literature as $H$ is not engulfed
in $G$ or $H$ is dense in the profinite topology of $G$.) For instance
if $G$ is an infinite
finitely presented group having property (T) but with no proper
finite index subgroups then we can take $H$ to be trivial. Then
form the amalgamation $A=G_1*_{H_1=H_2}G_2$ where $G_1$ and $G_2$ are isomorphic
copies of $G$ with $H_1,H_2$ the corresponding copies of the subgroup $H$.
Now $A$ is finitely presented and has an unbounded action on the tree given
by this splitting. But suppose that $A$ acts on a locally finite tree. Then
$G_1$ fixes a vertex $v_1$ of this tree (subdividing if necessary) and $G_2$
fixes some vertex $v_2$. On taking a path from $v_1$ to $v_2$, the finite
valence means that the subgroup $L_1$ of $G_1$ which fixes
the path between $v_1$ and $v_2$ has finite index in $G_1$.
Now $H=H_1=H_2$ is a subgroup of both $G_1$ and $G_2$ and
so fixes both $v_1$ and $v_2$, thus $H$ is contained in $L_1$. But our chosen
property of $H_1$ means that $L_1$ is equal to $G_1$. Thus $G_1$ fixes the
path between $v_1$ and $v_2$, so certainly fixes $v_2$. But so does $G_2$,
meaning that $A=\langle G_1,G_2\rangle$ does as well.

We now have a brief consideration of groups that are not finitely
generated. In our main theorem \ref{main}, if the initial metric space $X$
is itself a simplicial tree $T$ and the quasi-action is an action by
automorphisms on $T$ then it might seem that the conclusion would
give nothing new. Recall that by standard facts, if there exists a
loxodromic element in this action then there is a unique minimal
$G$-invariant subtree $T_G$ of $T$. (If there is no loxodromic element
then either we have a bounded action, or a type (2) action whereupon
orbits cannot be coarse connected.) Now if $G$
is finitely generated then (say by \cite{dd} Proposition I.4.13) the
quotient graph $G\backslash T_G$ is finite and so certainly this is a
cobounded action. In this case we know the orbits are coarse connected and
this generalises to groups that need not be finitely generated.
\begin{prop} \label{mint}
  Let $G$ be an arbitrary group acting by isometries on a simplicial
  tree $T$ and with a loxodromic element.
  Then the action of $G$ on the minimal invariant subtree $T_G$ 
  is cobounded if and only if the orbits are coarse connected.
\end{prop}
\begin{proof}
If the orbits are not coarse connected then the
action on $T$, on $T_G$, or indeed on any $G$-invariant geodesic subspace
of $T$ cannot be cobounded by Lemma \ref{ccnq} (ii). Now take any vertex $v$ in
$T_G$. The convex closure of
$Orb(v)$ must be equal to $T_G$ by minimality and Theorem \ref{sstt} tells
us that if $Orb(v)$  is coarse connected then it is coarse dense in its
convex closure $T_G$.
\end{proof}
Note that for $G$ acting by automorphisms on a tree $T$, saying the action is
cobounded is the same as saying that the quotient graph $G\backslash T$ is
a bounded graph, even if this graph is not locally finite.
\begin{co}
  Suppose that the group $G$ acts unboundedly on the locally finite
  simplicial tree $T$ by isometries with coarse connected orbits.
  Then the minimal invariant subtree has bounded valence.
\end{co}
\begin{proof}
  We can assume that $G$ acts by automorphisms as otherwise $T=\R$.
There will be a loxodromic element so the action of $G$ on $T_G$ is cobounded  
by Proposition \ref{mint}. But then there are only finitely many orbits
of vertices, each of finite valence, so $T$ has bounded valence.
\end{proof}

We finish this subsection by giving a counterexample to this if the group
(necessarily infinitely generated) does not have coarse connected orbits.
We will find a group $G$ with a type (2) action on a locally finite tree,
but where every action of $G$ on any bounded valence tree has a global
fixed point. This action will even
be metrically proper and it is a variation on Section 5 Example 2.\\
\hfill\\
{\bf Example 4}: Consider the restricted direct product
$G=C_2\times C_3\times C_5\times C_7\times \ldots$. By regarding $G$ as a
direct union of the increasing sequence of subgroups
$C_2,C_2,\times C_3,C_2\times C_3\times C_5,\ldots $ we can use the coset
construction as before to obtain a type (2) action of $G$ on a locally
finite tree but with unbounded valence. As all stabilisers are finite,
this is a proper action.

Now suppose $T$ is any tree with
valence bounded by $N$ and consider an element $g$ acting on $T$ with
$g$ having order coprime to the numbers 
$2,3,4,\ldots ,N$. Then $g$ must fix a vertex $v_0$ as
$T$ is a tree but it must then also fix the vertices adjacent to $v_0$
by Orbit - Stabiliser and so on, thus $g$ acts as the identity. This means
that if $p^+$ is the smallest prime which is greater than $N$
and $p^-$ the largest prime less than $N$
then any element
in the subgroup $C_{p^+}\times\ldots$ acts trivially on $T$. Thus the quotient
of $G$ by this subgroup which is in the kernel of the action leaves us
only the finite group $H=C_2\times C_3\times\ldots \times C_{p^-}$
and so any action of $G$ on such a $T$ will have bounded orbits and hence
a global fixed point.

\subsection{Metrically proper actions}
Throughout this paper we have been considering quasi-actions of arbitrary
groups with very weak conditions on the quasi-action, so this level of
generality does not allow us to identify the group from the quasi-action.
For instance in our earlier results, whenever a group $Q$ has a quasi-action
$\alpha$ satisfying our conditions and $Q=G/N$ is a quotient of some other
group $G$ then setting $\alpha'(g,x)=\alpha(gN,x)$ results in the quasi-action
$\alpha'$ of $G$ also satisfying these conditions.

  However we will end by showing that if in addition the quasi-action is
  metrically proper then this is enough. Moreover no finiteness condition
  is needed here on the space.
\begin{thm} \label{vfmn}
  Suppose that an arbitrary group $G$ has a metrically proper
quasi-action on some metric space where a quasi-orbit quasi-isometrically
  embeds in some simplicial tree.\\ 
  $\bullet$ If the quasi-orbits are not coarse connected then $G$ is
  not finitely generated.\\
  $\bullet$ If the quasi-orbits are coarse connected then $G$ is a
  finitely generated, virtually free group.
\end{thm}
\begin{proof}
  The division into the finitely and infinitely generated cases
  is Proposition \ref{fgn} (ii). Assuming therefore that 
  quasi-orbits are coarse connected, the quasi-orbits quasi-isometrically
  embed into trees by the hypothesis and they are also proper metric spaces,
  because the metrically proper condition implies that closed balls in
  any such quasi-orbit are finite sets.
  Therefore the conditions for Theorem \ref{main} apply. Consequently
  we obtain a reduction of $\alpha$ to a (quasi-)action $\beta$ say,
  whereupon $\beta$ is still metrically proper. 

  Running through the conclusions in turn, if $\beta$ is the trivial action
  on a point but also metrically proper then $G$ must be finite.

  If $\beta$ is a cobounded quasi-action on $\R$, so that quasi-orbits
  are quasi-isometric to $\R$ then by Theorem \ref{gnact} (ii) we can change
  this into an equivalent isometric action on a geodesic metric space $Z$
  that is quasi-isometric to $\R$. This action of $G$ will be both
  cobounded and metrically proper, so by Svarc - Milnor we have that the
  group $G$ is quasi-isometric to $Z$ and hence to $\R$. Thus $G$ has
  two ends which means that it is virtually cyclic (this result goes back
  to Hopf in \cite{hpf} from 1944).

  Otherwise $G$ acts coboundedly by isometries on a bounded valence
  bushy tree and hence also by automorphisms. This action is
  equivalent to $\alpha$ and so is metrically proper. Thus all stabilisers
  are finite. Coboundedness and bounded valence mean that the quotient must
  be a finite graph, whereupon \cite{ser} II.2.6 Proposition 11 states
  that the fundamental group of a finite graph of groups with all vertex
  groups finite is virtually a free group.
\end{proof}
Note that the above has the following consequence.
\begin{co}
  If $G$ is a finitely generated group which quasi-isometrically embeds in
  some simplicial tree $T$ then $G$ is virtually free.
\end{co}
\begin{proof}
First assume that $G$ is quasi-isometric to $T$. As $G$ acts on itself by
left multiplication, it quasi-acts on $T$ by Example \ref{qunew}. This
quasi-action is also metrically proper and cobounded, thus quasi-orbits
are quasi-isometric to $T$ (and certainly coarse connected) so Theorem
\ref{vfmn} applies.

If $G$ quasi-isometrically embeds in $T$ then $G$ is certainly coarse
connected (with the word metric/Cayley graph from a finite generating
set), thus its image under this embedding is coarse connected and hence
is also quasi-isometric to some tree by Theorem \ref{sstt}, so $G$ is too.
\end{proof}
This result is well known but would normally be proven using the Stallings
Theorem and Dunwoody's Accessibility Theorem (see \cite{drkp} Chapter 20
for a thorough treatment).
Our argument is not completely independent of these approaches because
the proof of Mosher, Sageev and Whyte's Theorem 1 in \ref{mswt} utilises
Dunwoody tracks. Nevertheless we have not relied on either of these
big results to conclude that our group is virtually free. We have relied
on Hopf's result in \cite{hpf} to cover the two ended case, but this seems
perfectly reasonable given its vintage.

\Address

\end{document}